\documentclass{article}

\usepackage{arxiv}

\PassOptionsToPackage{numbers,compress}{natbib}

\usepackage[utf8]{inputenc} 
\usepackage[T1]{fontenc}    
\usepackage{hyperref}       
\usepackage{url}            
\usepackage{booktabs}       
\usepackage{amsfonts}       
\usepackage{nicefrac}       
\usepackage{microtype}      
\usepackage{xcolor}         

\usepackage{microtype}
\usepackage{graphicx}
\usepackage{wrapfig}
\usepackage{subcaption}
\graphicspath{{figures/}}

\usepackage{amsmath}
\usepackage{amssymb}
\usepackage{mathtools}
\usepackage{amsthm}
\usepackage{pifont}
\newcommand{\xmark}{\ding{55}}
\usepackage{booktabs}

\usepackage{amsfonts}
\usepackage{bm}
\usepackage{microtype}
\usepackage{enumitem}

\usepackage[capitalize,noabbrev]{cleveref}

\theoremstyle{plain}
\newtheorem{theorem}{Theorem}[section]
\newtheorem{proposition}[theorem]{Proposition}
\newtheorem{lemma}[theorem]{Lemma}

\theoremstyle{definition}
\newtheorem{definition}[theorem]{Definition}

\theoremstyle{remark}
\newtheorem{remark}[theorem]{Remark}

\newcommand{\MEMO}[1]{}

\newcommand{\bbR}{\mathbb{R}}

\newcommand{\bbT}{\mathbb{T}}

\newcommand{\bbH}{\mathbb{H}}
\newcommand{\bbS}{\mathbb{S}}

\newcommand{\bbP}{\mathbb{P}}

\newcommand{\inner}[2]{\langle #1, #2 \rangle}
\newcommand{\F}{\mathcal{F}}
\newcommand{\Finve}{\mathcal{F}^{-1}}
\newcommand{\Hd}{\mathbb{H}_d}
\newcommand{\Hm}{\mathbb{H}_m}
\newcommand{\op}{\mathrm{op}}

\newcommand{\norm}[1]{\left\lVert #1\right\rVert}
\newcommand{\abs}[1]{\left\lvert #1\right\rvert}

\title{Symplectic Neural Operators for Learning Infinite Dimensional Hamiltonian Systems}

\author{%
Yeang Makara \\
Graduate School of Science\\
Kobe University\\
1-1 Rokkodai-cho, Nada-ku, Kobe, Japan\\ 
\texttt{244s025s@stu.kobe-u.ac.jp} \\
\And 
Yusuke Tanaka\\
NTT Communication Science Laboratories\\
2-4, Hikaridai, Seika-cho, Soraku-gun, Kyoto, Japan\\
\texttt{ysk.tanaka@ntt.com}
\And Takashi Matsubara\\
Faculty of Information Science and Technology\\
Hokkaido University\\
Kita 14, Nishi 9, Kita-ku, Sapporo, Hokkaido, Japan\\
\texttt{matsubara@ist.hokudai.ac.jp}
\And
Takaharu Yaguchi\\
Institute of Mathematics for Industry\\
Kyushu University\\
RIKEN AIP\\
744 Motooka, Nishi-ku, Fukuoka, Japan \\
\texttt{yaguchi@imi.kyushu-u.ac.jp} \\
}

\begin{document}

\maketitle

\begin{abstract}
  The modeling and simulation of infinite-dimensional Hamiltonian systems are central problems in mathematical physics and engineering, however they pose significant computational and structural challenges for standard data-driven architectures. In this work, we introduce the \textbf{Symplectic Neural Operator}, a neural operator architecture designed to preserve the symplectic structure intrinsic to Hamiltonian PDEs. We provide a theoretical characterization of their symplecticity and establish a rigorous long-term stability result based on the combination of symplectic structure preservation and learning accuracy. Numerical experiments on canonical Hamiltonian PDEs corroborate this theoretical result and show that SNOs exhibit improved energy behavior compared with non-structure-preserving neural operators.
\end{abstract}

\section{Introduction}
\label{sec:introduction}
\begin{wrapfigure}{r}{0.45\linewidth}
  \centering
  \includegraphics[width=\linewidth]{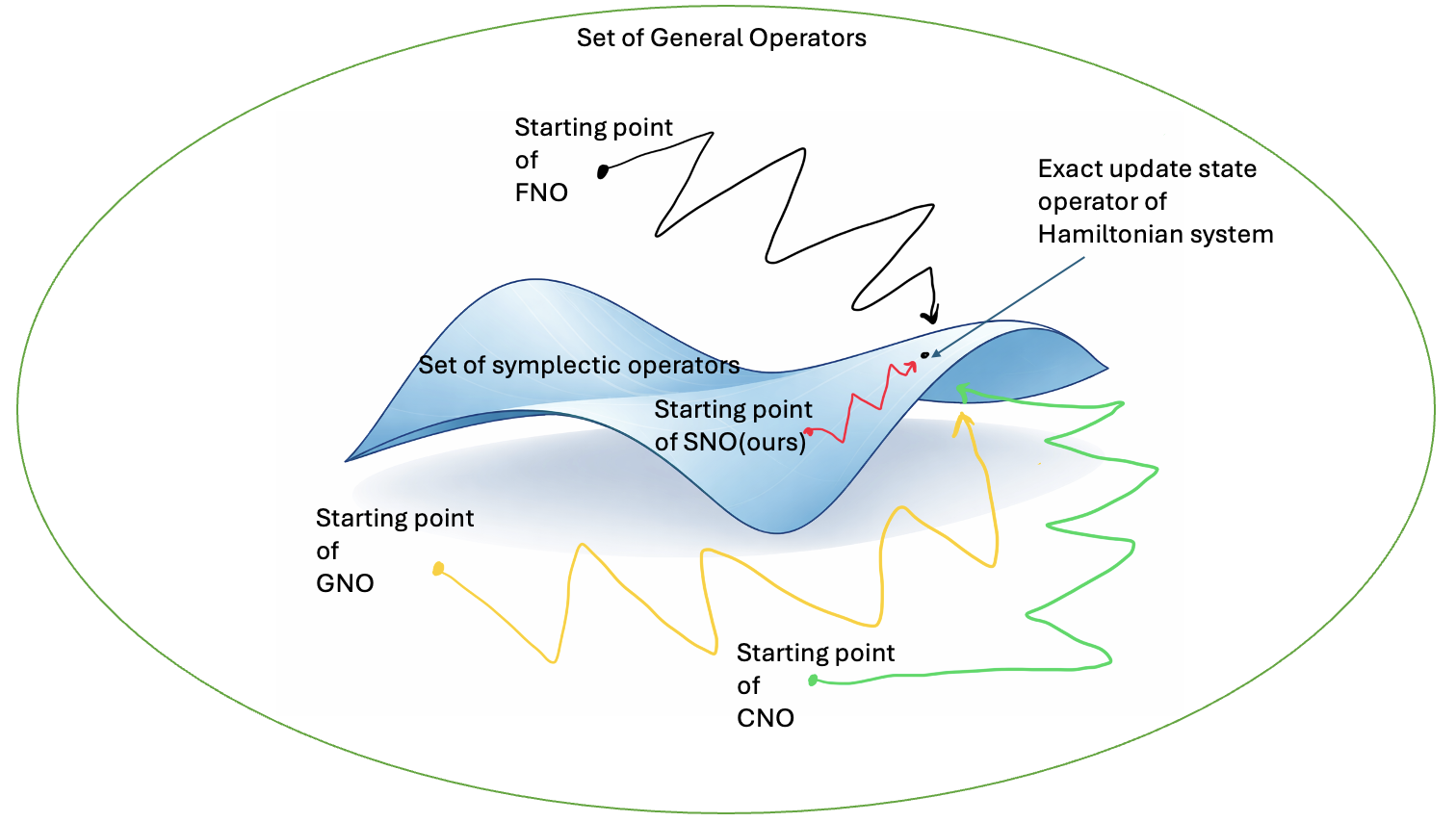}
  \caption{Efficient hypothesis class via symplectic restriction}
  \label{sno_copm}
\end{wrapfigure}
Hamiltonian mechanics provides a unifying language for conservative dynamics by combining an energy functional with a symplectic structure.
In finite dimensions, this framework underlies classical mechanical systems and yields flows that preserve qualitative properties such as phase-volume, time-reversibility, and long-time stability.
Crucially, many physical models of interest are \emph{inherently infinite-dimensional}: waves, fluids, plasmas, and field theories are naturally formulated as Hamiltonian PDEs, whose phase spaces are function spaces rather than Euclidean spaces \cite{marsden1999mechanics,bridges2001multisymplectic}.
In this setting, the time-$t$ evolution map is a \emph{symplectic operator} acting on functions, and this geometric constraint is not a cosmetic detail: it is central to faithful long-time dynamics \cite{hairerLubichWanner2006gni}.

Learning such systems from data is increasingly important for fast surrogate simulation, inverse problems, uncertainty quantification, and scientific discovery.
This goal naturally calls for models that can learn mappings \emph{between function spaces}.
Neural operators---including Fourier Neural Operator, Graph Neral Operator, Convolutional Neural Operator and DeepONet---have demonstrated strong performance for learning solution operators of PDE families by training directly on function-to-function mappings \cite{li2021fno,kovachki2023neuraloperator, li2020mgno, raonic2023cno, lu2021deeponet}.
Moreover, physics-regularized variants such as physics-informed neural operators can incorporate PDE constraints into operator learning \cite{li2021pino}.

However, for Hamiltonian systems, approximation accuracy alone is not enough.
Hamiltonian dynamics are governed not only by an evolution equation, but also by a symplectic structure that enforces geometric invariants and controls long-time behavior.
Classical results in geometric numerical integration show that \emph{symplecticity}, rather than pointwise energy matching, is the key property behind stable long-time rollouts \cite{hairerLubichWanner2006gni}.
Standard neural operators treat inputs/outputs as unconstrained functions and typically do not encode the phase-space geometry of the flow.
As a consequence, they can achieve low short-term error yet exhibit energy drift and instability under long-time rollout---a failure mode that becomes even more severe in high-dimensional or continuum regimes where small structural errors accumulate across many interacting modes.
\begin{table}
\centering
\footnotesize
\setlength{\tabcolsep}{4pt}
\renewcommand{\arraystretch}{1.0}
\caption{\small{Structural comparison of NOs.}}
\label{tab:model_comparison}
\begin{tabular}{lcc} \toprule Model & Operator & Symplectic \\ \midrule DeepONet & \checkmark & \xmark \\ FNO & \checkmark & \xmark \\ GNO & \checkmark & \xmark \\ CNO & \checkmark & \xmark \\ SympNet & \xmark & \checkmark \\ \midrule \textbf{SNO (ours)} & \checkmark & \checkmark\\ \bottomrule 
\end{tabular}
\end{table}
A large body of recent work addresses structure preservation in \emph{finite-dimensional} learned dynamics.
Hamiltonian Neural Networks learn a Hamiltonian and recover dynamics via the canonical symplectic form \cite{greydanus2019hnn}, while Lagrangian Neural Networks learn Lagrangians and avoid the need for canonical coordinates \cite{cranmer2020lnn}.
Other approaches directly parameterize symplectic maps, including SympNets and related symplectic architectures \cite{jin2020sympnets,zhong2020symoden}.
Despite their success, these models are not directly applicable to infinite-dimensional Hamiltonian PDEs in a mesh-independent way:
they typically operate on vectors in $\mathbb{R}^d$, depend on a specific discretization, and do not naturally define operators that generalize across resolutions or grids---which is a key requirement for PDEs surrogate modeling \cite{kovachki2023neuraloperator}.
\begin{table}[t]
\centering
\caption{Conceptual comparison between classical symplectic integrators and SNO.}
\begin{tabular}{lll}
\toprule
 & Classical symplectic integrator & Symplectic Neural Operator \\
\midrule
Input knowledge & \small{Hamiltonian/vector field known} & Trajectory data \\
Learned object & None & Finite-time flow map \(\Phi_H^\tau\) \\
\small{Accuracy controlled by} & Step size and method order & Learning error \(\varepsilon\) \\
Time interval & Built from small steps & Learns fixed \(\tau\)-flow directly \\
Symplecticity & By numerical scheme & By neural-operator architecture \\
\small{PDE implementation} & \small{May require discretization}  solves & \small{Direct learned operator evaluation} \\
Best use case & \small{Simulation from known equations} & Fast surrogate from data \\
\bottomrule
\end{tabular}
\end{table}\\
\textbf{Our approach.}
We introduce a \emph{symplectic neural operator} (SNO) for learning finite-time Hamiltonian flow maps in function spaces.\footnote{We make our code available at \url{https://anonymous.4open.science/r/SNO-216D}}
Rather than learning the Hamiltonian functional and then applying a separate numerical integrator, SNO directly learns the map
\(
    \Phi_H^\tau : (q(\cdot),p(\cdot)) \mapsto (q(\cdot,\tau),p(\cdot,\tau))
\)
from data.
The architecture is constructed as a composition of symplectic shear operators, thus the learned update map is symplectic by construction while retaining the discretization-invariant advantages of neural operators \cite{kovachki2023neuraloperator}.
This allows SNO to combine two complementary ideas: the geometric structure preservation of symplectic integration and the fast function-to-function surrogate modeling capability of neural operators.
The resulting model is especially attractive for repeated long-time rollouts,
backward-in-time reconstruction of Hamiltonian states, and uncertainty quantification,
where evaluating a trained operator can be substantially cheaper than repeatedly
solving the underlying Hamiltonian PDEs.\\
\textbf{Contributions.}
The main contributions of this paper are:
\begin{itemize}
  \item \textbf{Structure-preserving operator model.}
  We propose a neural-operator architecture for Hamiltonian PDEs whose learned fixed time update map is \emph{symplectic by construction}, yielding a structure-preserving function-to-function surrogate(see Table~\ref{tab:model_comparison}).
  \item \textbf{Efficient hypothesis class via symplectic restriction.}
  Because the exact flow map of a Hamiltonian system is symplectic, restricting learning to the set of symplectic operators reduces the search space relative to unconstrained neural operators.
  This acts as a strong inductive bias (physics-guided architectural prior) and can improve data efficiency and generalization(see Figure~\ref{sno_copm}).
  \item \textbf{Improved long-time rollout behavior.}
  We prove that if the learned update map is symplectic and constitutes a smooth $\varepsilon$-perturbation of the exact Hamiltonian flow (in the Hamiltonian sense), then it nearly preserves a modified Hamiltonian over long-time evolution. This provides a rigorous explanation for the controlled energy behavior and stability of SNOs.

  \item \textbf{Empirical validation on Hamiltonian PDE benchmarks.}
  We evaluate on representative Hamiltonian PDEs and compare against standard neural operators, demonstrating that symplectic structure in operator learning improves long-horizon fidelity while retaining competitive short-term accuracy.
\end{itemize}

\section{Infinite-Dimensional Hamiltonian Systems}
\label{sec:inf-dim-hamiltonian}

Many PDEs and continuum models can be viewed as
Hamiltonian dynamics on \emph{infinite-dimensional} phase spaces.
A systematic functional-analytic study of such systems, emphasizing the geometric meaning of
symplectic forms and Hamiltonian flows in infinite dimensions, appears already in the work of
Chernoff--Marsden and related developments in infinite-dimensional mechanics
\cite{chernoffMarsden1974properties}.
On the differential-geometric side, Weinstein formulated symplectic structures on Banach manifolds
and proved an infinite-dimensional strong Darboux theorem using Moser's method
\cite{weinstein1969symplecticBanach}.
A key subtlety is that in infinite dimensions the symplectic form can be \emph{weak} rather than \emph{strong},
and this distinction affects even local normal forms (Darboux charts), as shown by Marsden's counterexample
\cite{marsden1972darbouxFails} and later refinements such as Bambusi's criteria \cite{bambusi1999darbouxWeak}.

\subsection{Infinite-Dimensional Symplectic Manifolds}
\label{subsec:inf-dim-symplectic}

\begin{definition}[Weak/strong symplectic forms on Banach/Hilbert manifolds]
Let $M$ be a smooth manifold modeled on a Banach space (a Banach manifold).
A \emph{(smooth) $2$-form} $\omega$ on $M$ assigns to each $x\in M$ a continuous bilinear,
skew-symmetric map
\(
\omega_x : T_xM \times T_xM \to \mathbb{R},
\)
varying smoothly with $x$. It is \emph{closed} if $d\omega = 0$.
Define the bundle map
\(
\omega^\flat_x : T_xM \to T_x^*M, \omega^\flat_x(v) := \omega_x(v,\cdot).
\)
We call $\omega$ 
\emph{weakly nondegenerate} if $\omega^\flat_x$ is injective for all $x$, and \emph{strongly nondegenerate} if $\omega^\flat_x$ is a Banach-space isomorphism for all $x$.
A \emph{(weak/strong) symplectic manifold} is a pair $(M,\omega)$ with $\omega$ closed and weak/strongly nondegenerate.
\end{definition}
In finite dimensions, injective automatically implies bijective, so ``weak'' and ``strong''
coincide. In infinite dimensions they differ: $T_xM$ and $T_x^*M$ can be very different spaces, and
$\omega^\flat_x$ may fail to be onto even if it is one-to-one \cite{bambusi1999darbouxWeak}.
\begin{definition}[Hilbert symplectic form induced by an inner product]
Let $\bbP$ be a real Hilbert space with inner product $\langle \cdot,\cdot\rangle$.
Fix a bounded, invertible, skew-adjoint operator $J:\bbP\to\bbP$.
Define
\(
\omega(u,v) := \langle Ju, v\rangle, u,v\in\bbP.
\)
Then $\omega$ is a continuous skew-symmetric bilinear form, and it is \emph{strongly} nondegenerate because
$\omega^\flat = \mathcal{R}\circ J$, where $\mathcal{R}:\bbP\to\bbP^*$ is the Riesz isomorphism.
\end{definition}
\textbf{Canonical Hilbert phase space.}
A standard model is $\bbP = \bbH \oplus \bbH$ with $\bbH:= \bbH_q \cong \bbH_p$ and
\(
\langle (q,p),(q',p')\rangle_{\bbP} := \langle q,q'\rangle_{\bbH} + \langle p,p'\rangle_{\bbH},
J(q,p) := (-p,q).
\)
Then
\(
\omega\big((q,p),(q',p')\big) = \langle q,p'\rangle_{\bbH} - \langle p,q'\rangle_{\bbH},
\)
which is the infinite-dimensional analogue of $\sum_i dq_i\wedge dp_i$.
\subsection{Hamiltonian Vector Fields on Hilbert Symplectic Manifolds}
\label{subsec:inf-dim-ham-vectorfield}
\begin{definition}[Hamiltonian vector field]
Let $(M,\omega)$ be a \emph{strong} symplectic manifold modeled on a Banach/Hilbert space, and let
$H:M\to\mathbb{R}$ be smooth.
The \emph{Hamiltonian vector field} $X_H$ is defined by
\(\iota_{X_H}\omega = dH,\quad
\text{i.e.}\quad
\omega_x\big(X_H(x), v\big) = dH(x)[v]
\forall v\in T_xM.
\)
In the strong case, $\omega^\flat_x$ is invertible, so $X_H(x)$ exists and is unique for all $x$.
\end{definition}
\textbf{Hilbert-gradient form.}
On $\bbP$ with $\omega(u,v)=\langle Ju,v\rangle$, any $C^1$ functional $H:\bbP\to\mathbb{R}$ has a
gradient $\nabla H(u)\in\bbP$ characterized by
\(
dH(u)[v] = \langle \nabla H(u), v\rangle \quad \forall v\in\bbP.
\)
Then the Hamiltonian vector field satisfies
\(
\langle JX_H(u), v\rangle = \langle \nabla H(u), v\rangle \ \ \forall v
\Longrightarrow
JX_H(u)=\nabla H(u),
\)
hence
\(
X_H(u)=J^{-1}\nabla H(u).
\)
The \emph{Hamiltonian system} is the evolution equation
\(
\dot u(t) = X_H(u(t)) = J^{-1}\nabla H(u(t)).
\)

\begin{remark}[Gradients and Hessians in infinite-dimensional Hilbert spaces]
\label{rem:frechet-gradient-hessian}
In this paper the notions of gradient and Hessian are understood
in the infinite-dimensional, Fr\'echet-differentiable sense, and should not be confused with
their finite-dimensional matrix counterparts.
Let $\bbH$ be a real Hilbert space.
If $V:\bbH\to\mathbb{R}$ is Fr\'echet differentiable at $q$, its derivative
$DV(q)\in\bbH^*$ is a bounded linear functional.
By the Riesz representation theorem, there exists a unique element
$\nabla V(q)\in\bbH$ such that
\(
DV(q)[h] = \langle \nabla V(q),\, h\rangle  \forall h\in\bbH.
\)
The mapping $q\mapsto\nabla V(q)$ is therefore a vector-valued function on $\bbH$.
If $V$ is twice Fr\'echet differentiable, the Hessian
$D(\nabla V)(q):\bbH\to\bbH$ is a bounded linear operator defined as the Fr\'echet derivative
of the gradient map.
Unlike in finite dimensions, the self-adjointness of this operator is not automatic and must
be assumed or verified explicitly.
This assumption is essential for the symplecticity of the nonlinear shear maps defined above.
\end{remark}

\begin{proposition}[Energy conservation]
Let $u(t)$ solve $\dot u = X_H(u)$ on a symplectic manifold $(M,\omega)$ (weak or strong) on any time interval
where the solution exists. Then $H(u(t))$ is constant in $t$
\cite{chernoffMarsden1974properties}.
\end{proposition}
\begin{proposition}[Symplecticity of the flow]
Assume $(M,\omega)$ is symplectic and $H$ is smooth so that $X_H$ exists and generates a (local) flow $\Phi^t$.
Then $(\Phi^t)^*\omega=\omega$ wherever the flow is defined
\cite{chernoffMarsden1974properties,weinstein1969symplecticBanach}.
\end{proposition}
\subsection{What Remains of Darboux's Theorem in Infinite Dimensions}
\label{subsec:darboux-inf-dim}
In finite dimensions, Darboux's theorem says that every symplectic form is locally equivalent to the standard
constant form. In infinite dimensions, the situation splits dramatically depending on whether $\omega$ is strong
or weak.
\begin{theorem}[Strong Darboux theorem in Banach/Hilbert setting {\cite{weinstein1969symplecticBanach}}]
Let $(M,\omega)$ be a \emph{strong} symplectic Banach manifold. Then for every $x\in M$, there exist local
coordinates near $x$ in which $\omega$ becomes constant (i.e.\ locally equivalent to $\omega_x$).
\end{theorem}
\textbf{Weak case Darboux can fail.}
Marsden constructed a weak symplectic form on a Hilbert space for which no Darboux chart exists,
so \emph{weak symplectic} does not imply local constancy of $\omega$ in general \cite{marsden1972darbouxFails}.\\
\textbf{When can Darboux be recovered for weak forms?}
Bambusi introduced a ``classifying space'' (a Banach space canonically associated to $\omega_x$) and proved that
a weak symplectic form can be reduced to a constant one under a local constancy condition on these spaces,
together with suitable smoothness assumptions \cite{bambusi1999darbouxWeak}.
More recent work discusses additional settings (e.g.\ projective limits / Fr\'echet-type constructions) where Moser's
method may fail without strong hypotheses \cite{pelletier2021projectiveLimitDarboux}.

\section{Problem Setup}
\label{sec:learn-discrete-flow}
In this paper we focus on \emph{learning the time-$\tau$ flow map} of an
(infinite-dimensional) Hamiltonian system from data, rather than learning the
Hamiltonian vector field directly.
This viewpoint is particularly natural for PDEs, where the state at each time
is a function and the evolution can be seen as an operator acting on function
spaces.\\
\textbf{Continuous Hamiltonian Flow and Discrete-Time Sampling:}
\label{subsec:flow-and-sampling}
Let $H:\bbP\to\mathbb{R}$ be a smooth Hamiltonian such that the Hamiltonian
vector field $X_H$ generates a flow $\Phi^t:\bbP\to\bbP$:
\(
\frac{d}{dt}u(t)=X_H(u(t)),u(0)=u_0, u(t)=\Phi^t(u_0).
\)
Whenever it is defined, the flow satisfies the group property
\(
\Phi^{t+s}=\Phi^t\circ\Phi^s, \Phi^0=\mathrm{Id},
\)
and it is symplectic:
\(
(\Phi^t)^*\omega=\omega.
\)
Fix a time $\tau>0$. The \emph{time-$\tau$ flow map} is
the operator
\(
\Phi^\tau:\bbP\to\bbP,
u_{n+1}=\Phi^\tau(u_n),
t_n:=n\tau.
\)
Thus, learning discrete dynamics amounts to approximating the operator $\Phi^\tau$
from data. \\
\textbf{Learning the Flow Map as an Operator Regression Problem:}
Let $\mathcal{U}\subset\bbP$ be a set of admissible states (e.g.\ bounded-energy
states) on which the time-$h$ map is well-defined.
We consider a dataset of one-step pairs
\(
\mathcal{D}_N=\big\{(u^{(i)},v^{(i)})\big\}_{i=1}^N
\subset \mathcal{U}\times\mathcal{U},
v^{(i)}\approx \Phi^\tau\big(u^{(i)}\big).
\)
The goal is to construct a parametric family of maps
\(
\mathcal{N}_\theta:\mathcal{U}\to\mathcal{U},
\)
indexed by parameters $\theta\in\Theta$, such that $\mathcal{N}_\theta\approx \Phi^\tau$.
A common approach is empirical risk minimization:
\(
\theta^\star \in \arg\min_{\theta\in\Theta}\;
\frac{1}{N}\sum_{i=1}^N
\ell\big(\mathcal{N}_\theta(u^{(i)}),\, v^{(i)}\big),
\)
where $\ell$ is a loss functional. For instance, using the Hilbert norm on $\bbP$,
\(
\ell(\hat v,v)=\|\hat v-v\|_{\bbP}^2.
\)
Depending on the PDE and the desired regularity, one may also use Sobolev-type
losses (e.g.\ $\|\cdot\|_{\bbS}$ norms) to penalize high-frequency errors.\\
\textbf{Rollout (iterated prediction).}
Once a one-step model $\mathcal{N}_\theta$ is learned, we can predict multiple
steps by composition:
\(
\hat u_{n+1}=\mathcal{N}_\theta(\hat u_n),
\hat u_n = \mathcal{N}_\theta^{\,n}(u_0).
\)
A key difficulty is that small one-step errors can accumulate over long rollouts.
For Hamiltonian systems, structure preservation (especially symplecticity) often
plays a decisive role in preventing catastrophic drift in long-time prediction.

\section{Proposed Method}
\subsection{Symplectic Operators} 
\label{subsec:symplectic-operators}
\begin{definition}[Linear symplectic operator]
A bounded invertible linear operator $S:\bbP\to\bbP$ is called \emph{symplectic} if it preserves $\omega$:
\(
\omega(Su,Sv)=\omega(u,v)\quad \forall u,v\in\bbP.
\)
Equivalently,
\(
S^* J S = J.
\)
We denote the group of such operators by $\mathrm{Sp}(\bbP,J)$.
\end{definition}
\MEMO{We first show the following two important propositions to construct the proposed symplectic neural operators.}
\begin{proposition}
\label{prop:composition-sympl}
If $S,T\in \mathrm{Sp}(\bbP,J)$, then $ST\in \mathrm{Sp}(\bbP,J)$ and $S^{-1}\in \mathrm{Sp}(\bbP,J)$.
\end{proposition}
\begin{proposition}[Shear-type linear symplectic operators]
\label{prop:linear-shear}
Let $A:\bbH\to\bbH$ be bounded and self-adjoint. Define
\(
S_A(q,p):=(q,\ p + Aq).
\)
Then $S_A$ is symplectic. Likewise, if $B:\bbH\to\bbH$ is bounded and self-adjoint, then
\(
T_B(q,p):=(q + Bp,\ p)
\)
is symplectic.
\end{proposition}
These shear-type symplectic operators play a fundamental role in both the
theoretical structure of $\mathrm{Sp}(\bbP,J)$ and in practical constructions,
as they admit simple parametrizations and will later serve as building blocks
for symplectic neural operators.
We now extend the notion of symplecticity from linear operators to nonlinear maps
acting on infinite-dimensional phase spaces.
\begin{definition}[Nonlinear symplectic map / symplectomorphism]
Let $(M,\omega)$ be a (strong) symplectic Banach or Hilbert manifold.
A map
\(
\Phi : M \to M
\)
is called a \emph{(nonlinear) symplectic map} or \emph{symplectomorphism} if: $\Phi$ is a $C^1$ diffeomorphism of $M$;
and 
\(
\Phi^*\omega = \omega.
\)
\end{definition}
\textbf{Derivative characterization.}
Equivalently, $\Phi$ is symplectic if and only if for every $x\in M$ and all
$v,w\in T_xM$,
\(
\omega_{\Phi(x)}\big(D\Phi(x)v,\, D\Phi(x)w\big)
=
\omega_x(v,w),
\)
where $D\Phi(x):T_xM\to T_{\Phi(x)}M$ denotes the Fréchet derivative.\\
\textbf{Hilbert phase space formulation.}
A $C^1$ diffeomorphism $\Phi:\bbP\to\bbP$ is symplectic if and only if
\(
(D\Phi(u))^*\, J\, D\Phi(u) = J, 
\forall u\in\bbP,
\)
where $(D\Phi(u))^*$ denotes the Hilbert adjoint.
\begin{proposition}[Shear-type nonlinear symplectic operators]
\label{prop:nonlinear-shear}
Let $V:\bbH\to\mathbb{R}$ be $C^2$ (Fr\'echet) and assume the Hessian operator
$D(\nabla V)(q):\bbH\to\bbH$ is self-adjoint for every $q$.
Define the nonlinear shear
\(
\Phi_V(q,p):=\big(q,\ p + \nabla V(q)\big).
\)
Then $\Phi_V$ is a symplectomorphism on its domain.
Similarly, if $W:\bbH\to\mathbb{R}$ is $C^2$ with self-adjoint $D(\nabla W)(p)$, then
\(
\Psi_W(q,p):=\big(q+\nabla W(p),\ p\big)
\)
is symplectic.
\end{proposition}
The explicit and local structure of shear-type nonlinear symplectic maps makes
them particularly suitable for parametrization by neural operators, where
$\nabla V$ and $\nabla W$ can be approximated by operator-valued neural networks
while preserving symplecticity by construction.
\subsection{Self-adjoint Fourier Neural Operators}
In this subsection, we introduce a new class of neural operators, called
\emph{Self-adjoint Fourier Neural Operators} (SAFNOs),
which serve as fundamental building blocks for linear symplectic neural operators.
Unlike the finite-dimensional setting, where a symmetric linear map can
always be constructed by an explicit algebraic symmetrization such as
$A \mapsto \tfrac12(A + A^\top)$, enforcing self-adjointness for operators acting on
function spaces is substantially more delicate.
In infinite dimensions, self-adjointness is not merely an algebraic property of a matrix
representation, but an operator-theoretic property that depends on the choice of function
space, inner product, operator domain, and boundary conditions.
In particular, the adjoint of a linear operator on a function space cannot, in general,
be computed as simply as taking the transpose of a matrix, and may fail to exist or coincide
with the original operator unless these structures are carefully aligned.
For these reasons, the design of self-adjoint neural operators must be carried out directly
at the level of function spaces, rather than imposed \emph{post hoc} through
finite-dimensional discretizations.
The SAFNO framework introduced below addresses this challenge by enforcing self-adjointness
through explicit spectral constraints on the Fourier symbols, yielding operators that are
provably self-adjoint on the target Hilbert space and stable under mesh refinement.
We begin by setting up the underlying function spaces and introducing preliminary tools
needed for the construction of our models and for proving their self-adjointness.
We write the $n$-torus as
\(
\mathbb{T}^n := (\mathbb{R}/\mathbb{Z})^n,
\)
endowed with the Haar probability measure $dx$.
\begin{theorem}[Unitarity of the Fourier transform on $\mathbb{T}^n$]\label{uf}
Define
\(
\F:L^2(\mathbb{T}^n;\mathbb{C}^m)\to \ell^2(\mathbb{Z}^n;\mathbb{C}^m),\)
\(
(\F v)(k):=\int_{\mathbb{T}^n} v(x)e^{-2\pi i k\cdot x}\,dx,
\)
where the integral is taken componentwise.
Then $\F$ is a unitary operator. Its inverse is given by the Fourier series
\(
(\Finve \hat v)(x)=\sum_{k\in\mathbb{Z}^n}\hat v(k)e^{2\pi i k\cdot x},
\)
with convergence in 
\(L^2(\mathbb{T}^n;\mathbb{C}^m),
\)
and $\Finve=\F^*$.
\end{theorem}
We can prove this theorem by using Parseval identity on $\bbT^n$; see, for example,
\cite{ReedSimonI, RieszNagy,EinsiedlerWard2017,Borthwick2020}.
\begin{theorem}[Self-Adjoint Linear Fourier Neural Operator]\label{safno}
Let
$\mathbb{T}^n := (\mathbb{R}/\mathbb{Z})^n$
be the $n$-dimensional torus equipped with the Haar probability measure.
Define the Hilbert spaces
$
\Hd := L^2(\mathbb{T}^n;\mathbb{C}^d),
\quad
\Hm := L^2(\mathbb{T}^n;\mathbb{C}^m),
$ 
with the complex Hilbert inner products.
Let
$
\F : \Hm \to \ell^2(\mathbb{Z}^n;\mathbb{C}^m)
$
denote the Fourier transform on $\mathbb{T}^n$.
Let $P\in\mathbb{C}^{m\times d}$ and for each $k=0,\dots,L-1$, let $W_k\in\mathbb{C}^{m\times m}$ satisfy the adjoint pair symmetry $W_k = W_{L-1-k}^*$
and $R_k:\mathbb{Z}^n\to\mathbb{C}^{m\times m}$ satisfy
      $\sup_{\xi\in\mathbb{Z}^n}\|R_k(\xi)\|_{\op}<\infty$ and the adjoint pair symmetry $R_k(\xi)=R_{L-1-k}(\xi)^*$
      for all $\xi\in\mathbb{Z}^n$.
Define linear operators $T_k:\Hm\to\Hm$ by
\(
(T_k v)(x)
:= W_k v(x)
 + \Finve\!\big(\xi\mapsto R_k(\xi)\,(\F v)(\xi)\big)(x),
\)
and define the $L$-layer linear Fourier Neural Operator $G:\Hd\to\Hd$ by
\(
v_0 = P a,
\quad
v_{k+1}=T_k v_k \ (k=0,\dots,L-1),
\quad
G(a)=P^* v_L.
\)
Then $G$ is a bounded self-adjoint operator on $\Hd$, i.e.
\(
\inner{G(a)}{b}_{\Hd}=\inner{a}{G(b)}_{\Hd}
\text{ for all }a,b\in\Hd.
\)
\end{theorem}
\subsection{Gradient Neural Functional}
The purpose of this subsection is to construct nonlinear operators whose
Fréchet derivatives are self-adjoint.
Such operators play a central role in symplectic neural architectures,
since Hamiltonian vector fields are generated as symplectic gradients
of scalar energy functionals.
In particular, if a nonlinear functional admits a gradient structure,
then its linearization automatically yields a self-adjoint operator,
which can be implemented using the SAFNOs introduced above.
Let $\mathbb{H}$ be a real Hilbert space with inner product $\inner{\cdot}{\cdot}_\mathbb{H}$.
Let functional
\(
\mathcal{E}:\mathbb{H}\to\mathbb{R}
\)
be (Fréchet) differentiable at $u\in\mathbb{H}$.
By the Riesz representation theorem, there exists a unique element
$\nabla\mathcal{E}(u)\in\mathbb{H}$ satisfying 
\(
D\mathcal{E}(u)[h]
=
\inner{\nabla\mathcal{E}(u)}{h}_\mathbb{H}
\quad
\forall h\in\mathbb{H}.
\)
The mapping
\(
\nabla\mathcal{E}:\mathbb{H}\to\mathbb{H}
\)
is called the \emph{gradient operator} associated with $\mathcal{E}$.\\
\textbf{Self-adjointness of the Jacobian.}
Assume that $\mathcal{E}\in C^2(\mathbb{H};\mathbb{R})$.
Then the Fréchet derivative of the gradient,
\(
D(\nabla\mathcal{E})(u):\mathbb{H}\to\mathbb{H},
\)
satisfies
\(
\inner{D(\nabla\mathcal{E})(u)h}{k}_\mathbb{H}
=
\inner{h}{D(\nabla\mathcal{E})(u)k}_\mathbb{H}
\quad
\forall h,k\in\mathbb{H},
\)
i.e. the Jacobian of $\nabla\mathcal{E}$ is a self-adjoint operator.
This identity follows from symmetry of second derivatives:
\(
D^2\mathcal{E}(u)[h,k]
=
D^2\mathcal{E}(u)[k,h].
\)
Hence, gradient mappings naturally induce self-adjoint linearizations.\\
\textbf{Neural parameterization of energy functionals.}
Let $\Hd=L^2(\mathbb{T}^n;\mathbb{R}^d)$.
We define a \emph{gradient neural functional} by specifying a scalar-valued energy
$
\mathcal{E}_\theta:\Hd\to\mathbb{R},
$
parameterized by trainable parameters $\theta$,
and setting the nonlinear operator
\(
\mathcal{G}_\theta := \nabla \mathcal{E}_\theta : \Hd \to \Hd.
\)
In this work, we consider energies of the form
\(
\mathcal{E}_\theta(u)
=
\int_{\mathbb{T}^n} \Phi_\theta\big( (Ku)(x) \big)\,dx,
\)
where
 $K:\Hd\to\Hm$ is a bounded linear operator, implemented as a SAFNO
and $\Phi_\theta:\mathbb{R}^m\to\mathbb{R}$ is a smooth scalar-valued neural network
      applied pointwise.
The gradient is then given by
\(
\nabla\mathcal{E}_\theta(u)
=
K^*\big( \nabla \Phi_\theta(Ku) \big),
\)
where $\nabla\Phi_\theta$ is the Euclidean gradient of $\Phi_\theta$.\\
\textbf{Self-adjoint linearization.}
Let $u\in\Hd$ and $h\in\Hd$.
The Fréchet derivative of $\mathcal{G}_\theta$ at $u$ is
\(
D\mathcal{G}_\theta(u)h
=
K^* \big( D^2\Phi_\theta(Ku)\, (K h) \big),
\)
where $D^2\Phi_\theta$ denotes the Hessian matrix of $\Phi_\theta$.
Since $D^2\Phi_\theta(Ku)$ is symmetric pointwise and $K^*$ is the adjoint of $K$,
it follows that
\(
D\mathcal{G}_\theta(u)
=
K^*\, M_{H_\theta(u)}\, K
\)
is a self-adjoint operator on $\Hd$,
where $M_{H_\theta(u)}$ denotes multiplication by the symmetric matrix-valued function
$H_\theta(u)(x)=D^2\Phi_\theta((Ku)(x))$.

\begin{theorem}[Gradient formula and self-adjoint linearization]\label{nsafno}
Let $\Hd=L^2(\mathbb{T}^n;\mathbb{R}^d)$ and $\Hm=L^2(\mathbb{T}^n;\mathbb{R}^m)$, and let $K:\Hd\to\Hm$ be a bounded linear operator with adjoint $K^*$. Let $\Phi_\theta:\mathbb{R}^m\to\mathbb{R}$ be $C^2$ with $\sup_{z}\|D^2\Phi_\theta(z)\|_{\mathrm{op}}<\infty$. Define energy functional 
\(
\mathcal{E}_\theta(u)=\int_{\mathbb{T}^n}\Phi_\theta\big((Ku)(x)\big)\,dx.
\)
Then $\mathcal{E}_\theta$ is Fr\'echet differentiable with
\(
\nabla\mathcal{E}_\theta(u)=K^*\big(\nabla\Phi_\theta(Ku)\big),
\)
and $\mathcal{G}_\theta:=\nabla\mathcal{E}_\theta$ is Fr\'echet differentiable with
\(
D\mathcal{G}_\theta(u)h = K^*\!\left(D^2\Phi_\theta(Ku)\,(Kh)\right), h\in\Hd,
\)
equivalently $D\mathcal{G}_\theta(u)=K^* M_{H_\theta(u)} K$, where $H_\theta(u)(x)=D^2\Phi_\theta((Ku)(x))$; in particular, $D\mathcal{G}_\theta(u)$ is a bounded self-adjoint operator on $\Hd$ for every $u\in\Hd$.
\end{theorem}

\subsection{The Architecture of Symplectic Neural Operator}
\label{subsec:arch-sympno}
We construct a learnable fixed time flow map
as a composition of shear-type symplectic blocks parameterized by the gradient neural functionals introduced in the previous subsection
.\\
\textbf{Gradient shear blocks.}
Let $\mathcal{E}^{q}_{\theta}:\Hd\to\mathbb{R}$ and $\mathcal{E}^{p}_{\vartheta}:\Hd\to\mathbb{R}$
be neural energy functionals of the form constructed above, and define the associated
gradient operators
\(
F^{q}_{\theta}:=\nabla \mathcal{E}^{q}_{\theta}:\Hd\to\Hd,
F^{p}_{\vartheta}:=\nabla \mathcal{E}^{p}_{\vartheta}:\Hd\to\Hd.
\)
We define the corresponding shear maps on $\mathbb{P}$ by
\(
\mathcal{T}^{\mathrm{up}}_{\theta}(q,p)
:=
\big(q+F^{q}_{\theta}(p),\,p\big),
\mathcal{T}^{\mathrm{low}}_{\vartheta}(q,p)
:=
\big(q,\,p+F^{p}_{\vartheta}(q)\big).
\)
Each block is symplectic, and compositions of these blocks remain symplectic.

\begin{definition}[Symplectic Neural Operator]
\label{def:sympno}
Fix an integer $L\ge 1$ and a parameter collection
\(
\Theta := (\theta_1,\vartheta_1,\dots,\theta_L,\vartheta_L).
\)
The \emph{Symplectic Neural Operator} is the map $\mathrm{SNO}_\Theta:\mathbb{P}\to\mathbb{P}$
defined by the alternating composition
\(
\mathrm{SNO}_\Theta
:=
\mathcal{T}^{\mathrm{up}}_{\theta_L}\circ \mathcal{T}^{\mathrm{low}}_{\vartheta_L}
\circ \cdots \circ
\mathcal{T}^{\mathrm{up}}_{\theta_1}\circ \mathcal{T}^{\mathrm{low}}_{\vartheta_1}.
\)
The specific ordering of upper and lower shear blocks is a modeling choice.
\end{definition}
\subsection{Long-time Stability of Symplectic Neural Operators}
We next clarify the sense in which the proposed SNO improves long-time stability.
A Hamiltonian flow preserves two different structures: the Hamiltonian energy and the symplectic
form. The exact flow \(\Phi^\tau\) satisfies
\(
    H(\Phi^\tau(u)) = H(u)
\)
and    
\(
    (\Phi^\tau)^*\omega = \omega.
\)
An SNO is designed to preserve the second property exactly. It does not, in general, preserve the
Hamiltonian \(H\) exactly.  The advantage of SNO is that
the learned error is constrained to be symplectic. Under a Hamiltonian perturbation
representation, this yields near-conservation of a modified Hamiltonian.
\begin{theorem}[Modified-Hamiltonian stability of SNO]\label{thm:sno_modified_hamiltonian}
Let $\mathcal{U}\subset\bbP$ be open and let $\mathcal{K}\Subset\mathcal{U}$ be a rollout
region. Assume that
$H\in C^{m+2}(\mathcal{U};\bbR)$, and that the exact flow $\Phi_H^t$ is well-defined on $\mathcal{U}$ for
$0\le t\le \tau$, and that the $\varepsilon$-error SNO learned map $S_\varepsilon$ admits the Hamiltonian
perturbation representation
\(
    S_\varepsilon=\Phi_{K_\varepsilon}^1\circ \Phi_H^\tau,
    K_\varepsilon=\sum_{j=1}^m \varepsilon^j K_j+\mathcal{O}_{C^{m+1}}(\varepsilon^{m+1}),
\)
where $K_j\in C^{m+2-j}(\mathcal{U};\bbR)$. Suppose moreover that the BCH logarithm
$\log(\exp(\tau\mathcal{L}_H)\exp(\mathcal{L}_{K_\varepsilon}))$ admits a Hamiltonian expansion up to order $m$
with remainder $\mathcal{O}(\varepsilon^{m+1})$ uniformly on $\mathcal{K}$, and that the required
Hamiltonian vector fields and Lie derivatives are bounded on $\mathcal{K}$. If
$u_{n+1}=S_\varepsilon(u_n)$, $u_0\in\mathcal{K}$, and $u_n\in\mathcal{K}$ for $0\le n\le N$, then there
exists a modified Hamiltonian
\(
    \widetilde H_{\varepsilon,m}
    =H+\varepsilon H_1+\varepsilon^2H_2+\cdots+\varepsilon^mH_m
\)
such that, for constants independent of $n$ and sufficiently small $\varepsilon$,
\(
    \|\widetilde H_{\varepsilon,m}-H\|_{C^0(\mathcal{K})}\le C\varepsilon,
    |\widetilde H_{\varepsilon,m}(u_n)-\widetilde H_{\varepsilon,m}(u_0)|
    \le Cn\varepsilon^{m+1}.
\)
Consequently,
\(
    |H(u_n)-H(u_0)|\le C\varepsilon+Cn\varepsilon^{m+1},
\)
and in particular $|H(u_n)-H(u_0)|\le C\varepsilon$ for $0\le n\le c\varepsilon^{-m}$.
\end{theorem}
\textbf{Advantage over classical symplectic integrators.}
Classical symplectic integrators, such as splitting methods, leapfrog schemes, and variational
integrators, provide powerful tools when the Hamiltonian functional and the associated variational
derivatives are explicitly known. Their main objective is to approximate the Hamiltonian vector
field by repeatedly applying small time steps while preserving a discrete symplectic structure.
In contrast, our setting is data-driven: the Hamiltonian functional, the exact vector field, or the
appropriate structure-preserving discretization may not be available. The SNO therefore learns the
finite-time flow map directly,
\(
    \Phi^\tau_H:\mathbb{P}\to\mathbb{P},
\)
from input--output state pairs, while enforcing symplecticity at the operator level.

This distinction leads to an important practical advantage. The accuracy of a classical integrator is
controlled by the step size and the order of the numerical method, whereas the accuracy of an SNO is
controlled by the learned flow-map error
\(
    \varepsilon
    =
    \sup_{z\in\mathcal{K}}
    \|S_\theta(z)-\Phi^\tau_H(z)\|.
\)
Thus, the learned time interval \(\tau\) need not be interpreted as a small numerical time step.
Even for a large sampling interval, the model can in principle achieve small error if the training data
and model capacity are sufficient. In this sense, SNO acts as a structure-preserving surrogate for the
exact time-\(\tau\) flow rather than as a conventional local time-stepping scheme.

Moreover, many structure-preserving integrators for Hamiltonian PDEs require repeated evaluations
of variational derivatives, careful treatment of boundary conditions, and in some cases implicit
solves at every time step. By contrast, once trained, an SNO computes the solution via a single forward pass of the learned symplectic map. This makes it attractive as a fast surrogate for repeated long-time
rollouts, inverse problems, and uncertainty quantification, where the same dynamical system must be
evaluated many times.

\section{Experimental Result}
We evaluate the proposed shear-type neural operator models on four benchmark
Hamiltonian PDEs: simple wave, electromagnetic wave, Schr\"odinger, and
Klein--Gordon.
\begin{figure}[htbp]
  \centering
  \includegraphics[width=\linewidth]{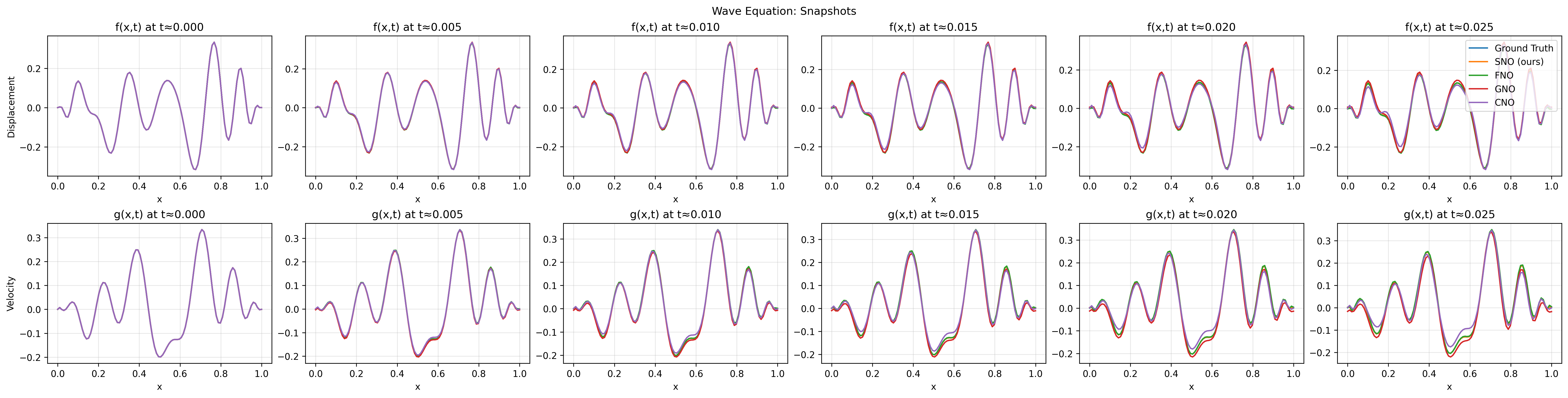}
  \caption{Wave equation: snapshots of displacement 
  and velocity at first 25 time-steps.}
  \label{fig:sw-sno-snapshots25}
\end{figure}
\begin{figure}[htbp]
  \centering
  \includegraphics[width=\linewidth]{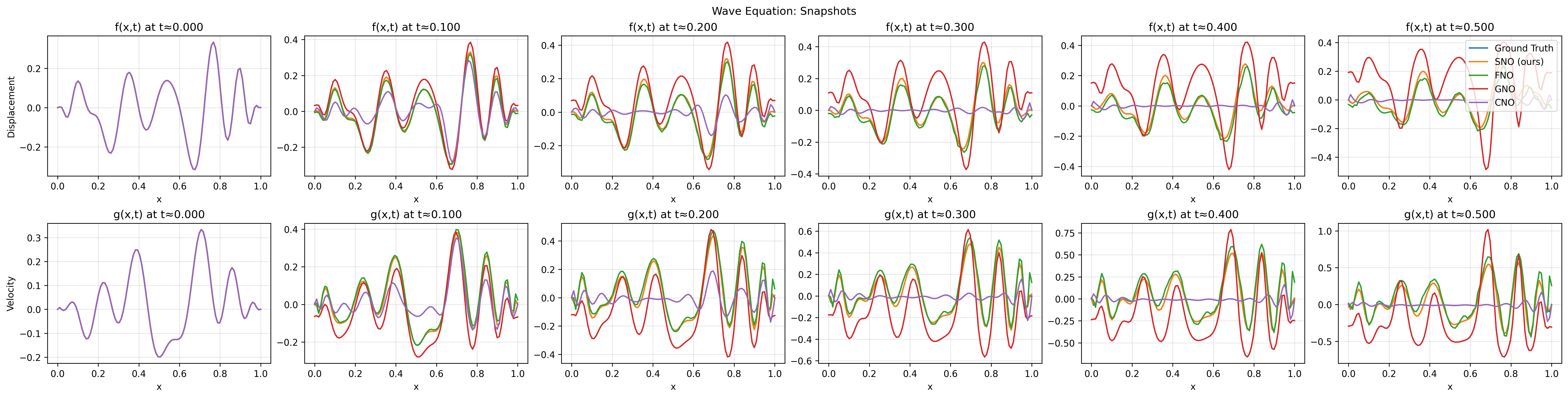}
  \caption{Wave equation: snapshots of displacement  and velocity  at first 500 time-steps.}
  \label{fig:sw-sno-snapshots}
\end{figure}

\begin{figure}[htbp]
  \centering
  \includegraphics[width=\linewidth]{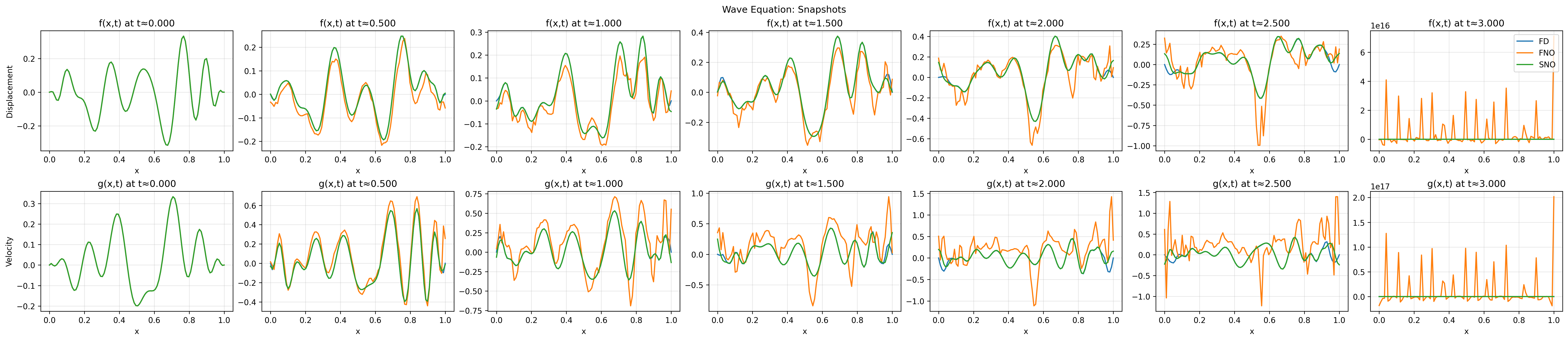}
  \caption{Wave equation: snapshots of displacement
  and velocity at first 3000 time-steps.}
  \label{fig:sw-sno-snapshots3000}
\end{figure}

\begin{figure}[htbp]
  \centering
  \includegraphics[width=0.58\linewidth]{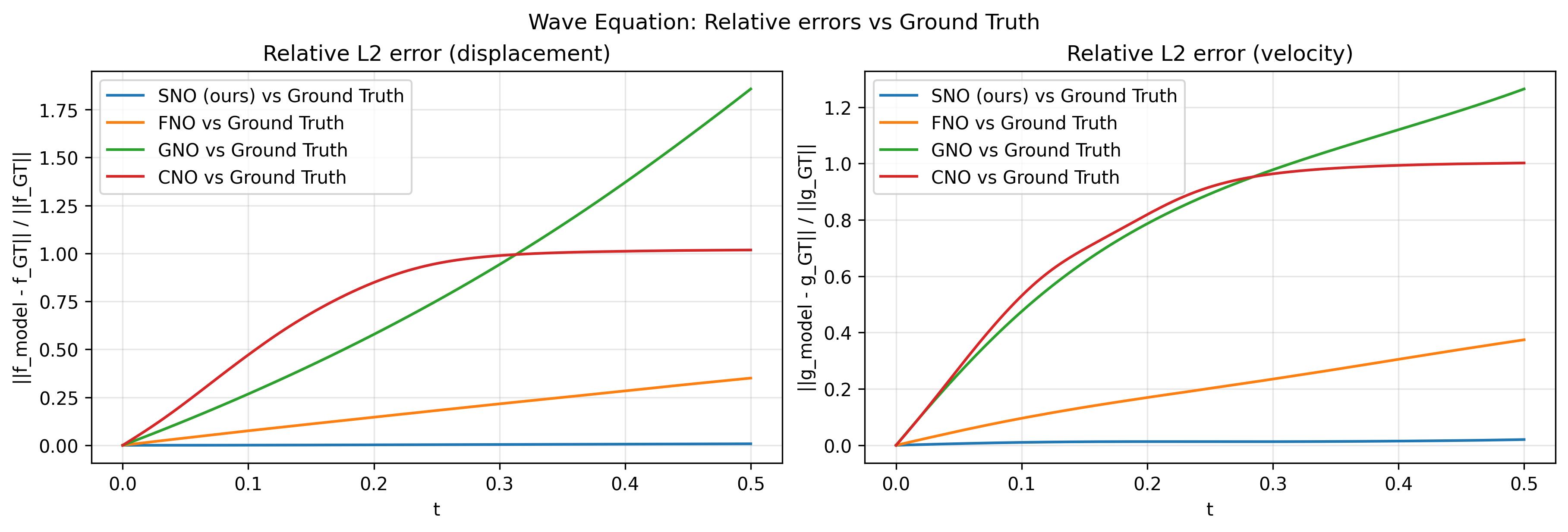}
  \includegraphics[width=0.3\linewidth]{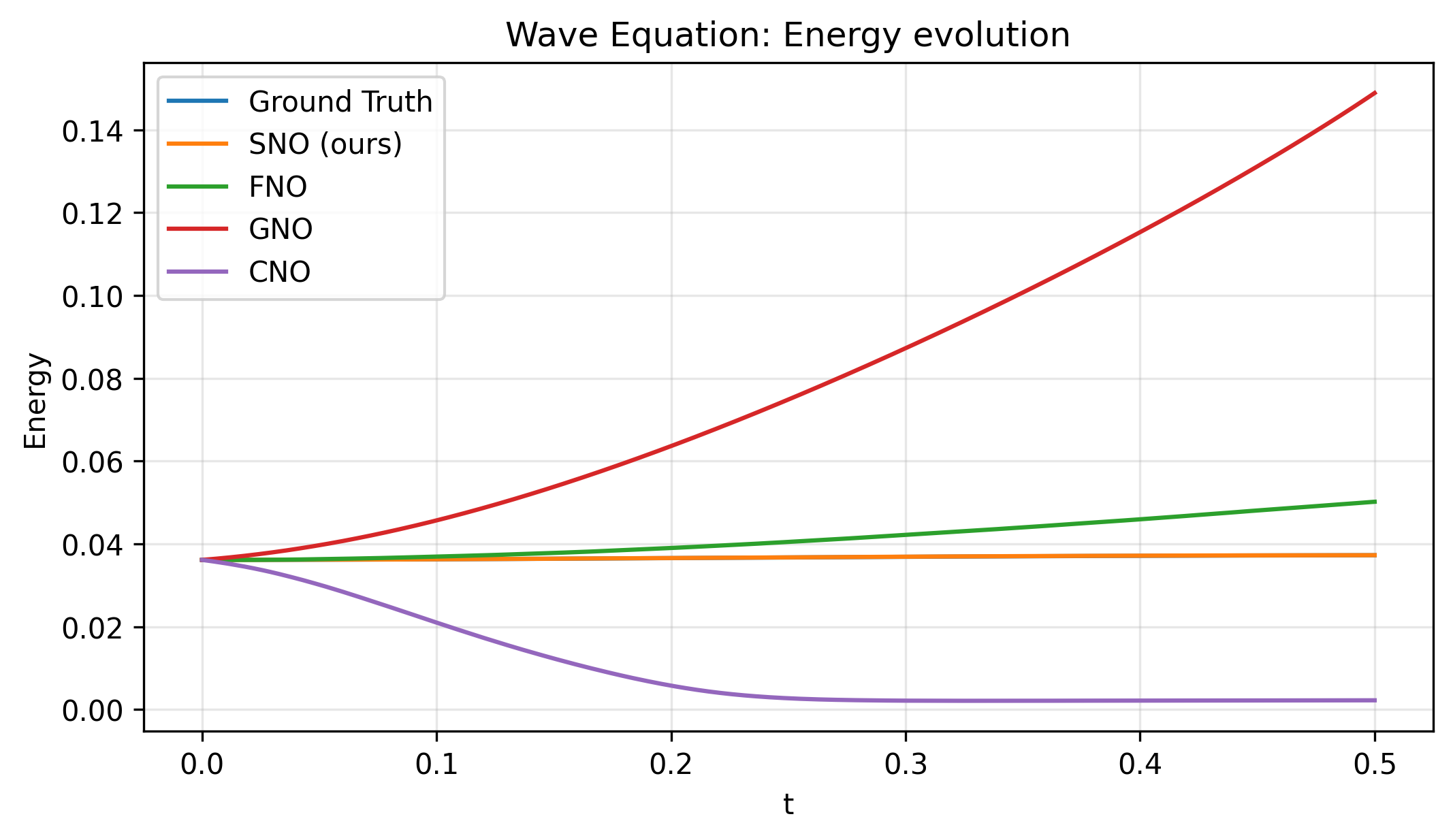}
  \caption{Wave equation: relative $L^2$ rollout errors against the ground truth and and energy evolution.}
  \label{fig:sw-sno-relerr}
\end{figure}

\begin{figure}[htbp]
  \centering
  \includegraphics[width=\linewidth]{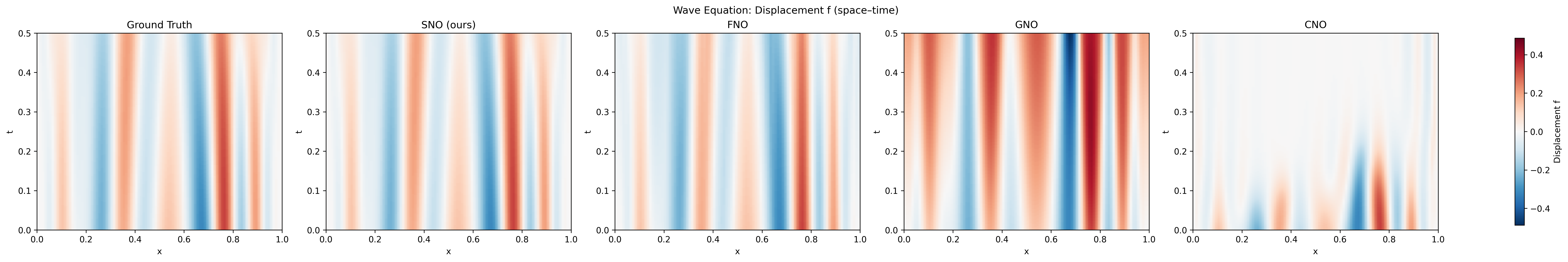}
  \includegraphics[width=\linewidth]{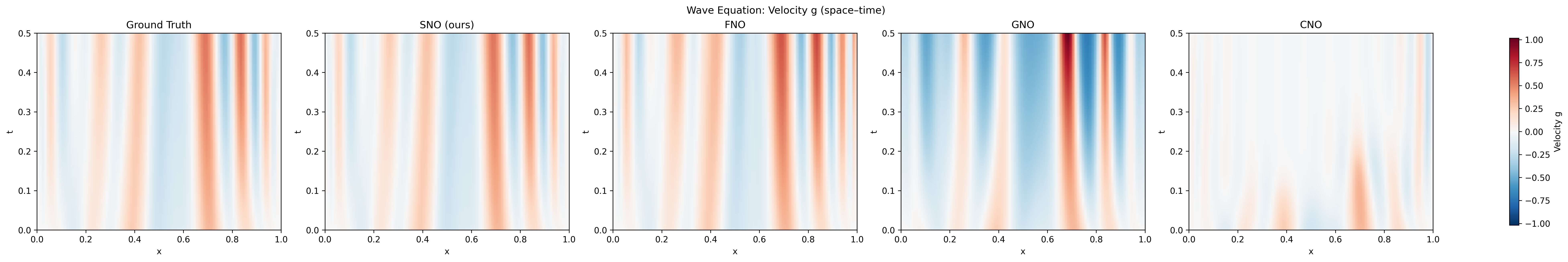}
  \caption{Wave equation: space--time displacement field and velocity field.}
  \label{fig:sw-sno-spacetime}
\end{figure}

\textbf{Experiment on wave equation:} In the first experiment, we focus on learning and predicting the solutions of a wave equation on the spatial interval $I=[0,1]$ with a wave speed of 
$c=0.05$. The Hamiltonian density of the system is $\mathcal{H}=\frac{1}{2}(p^2+c^2q_x^2)$ where $q$ and $p$ stand for displacement and velocity respectively. 
Figures~\ref{fig:sw-sno-snapshots25}, 
\ref{fig:sw-sno-snapshots},
and 
\ref{fig:sw-sno-spacetime}
compare the predicted solutions of the wave equation 
produced by the proposed SNO with several baseline neural operators.
At short times (within the first 10–20 time steps), 
all models produce visually indistinguishable predictions that closely match the ground truth. 
This behavior is expected, 
as short-time evolution is mainly governed by the local accuracy of the learned one-step map, 
for which standard neural operators are already sufficient.
In contrast, substantial differences emerge in the long-time regime. 
As the rollout horizon increases to several hundreds of time steps, 
baseline models such as FNO, GNO, and CNO exhibit a gradual accumulation of phase and amplitude errors. 
These errors manifest as spurious oscillations, amplitude drift, 
and a loss of coherence in both the displacement and velocity fields.
By comparison, the proposed SNO maintains stable oscillatory behavior and 
preserves the qualitative structure of the solution over long times.
This improved long-time performance can be attributed to the symplectic structure 
embedded in the SNO architecture. 
By preserving the underlying Hamiltonian geometry of the wave equation, 
the SNO mitigates the systematic error accumulation that typically arises 
during long rollouts of non–structure-preserving models. 
As a result, while all models perform comparably in the short term, 
the SNO exhibits superior stability and fidelity in long-time prediction.
Among the baseline methods, FNO demonstrates the second-best long-time stability, 
as shown in Figure~\ref{fig:sw-sno-snapshots}. At 500 time steps, its prediction remains 
relatively close to both the SNO and the ground truth. 
However, by 3000 time steps, significant deviations become apparent, 
as illustrated in Figure~\ref{fig:sw-sno-snapshots3000}.\\
Figure~\ref{fig:sw-sno-relerr} reports the evolution of the Hamiltonian energy. 
While the baseline models exhibit a clear monotonic energy drift and eventual instability, 
the SNO maintains an energy profile that remains nearly constant and closely aligned with 
the ground truth throughout the entire rollout horizon.\\
This enhanced stability is further reflected in the relative $L^2$ rollout errors also 
shown in Figure~\ref{fig:sw-sno-relerr}. For both displacement and velocity, 
the errors of FNO, CNO, and GNO grow rapidly over time, whereas the SNO errors 
remain small and increase only mildly, indicating controlled and stable long-time behavior.
These results are consistent with the theoretical
analysis in previous section: symplectic structure does
not merely improve short-term accuracy, but controls the accumulation of error
over long rollouts.

\textbf{Further experimental results.}
Similar trends are observed for the electromagnetic wave equation, the
Schr\"odinger equation, and the nonlinear Klein--Gordon equation. Due to space
constraints, these additional experiments are reported in the appendix.
\section{Limitations and Future Work}
\label{sec:limitations}
The proposed SNO has a few limitations. First, because the architecture is built from compositions of shear-type symplectic operator layers, it requires longer  training time compared with general neural operators. Second, the present formulation applies directly to Hamiltonian systems modeled on strong canonical symplectic Hilbert phase spaces; extending SNOs to Hamiltonian systems on general symplectic manifolds equipped with weak symplectic forms, will require more sophisticated geometric constructions and additional functional-analytic theory. Finally, our modified-Hamiltonian stability result relies on some assumptions, including the Hamiltonian perturbation representation of the learned error map and the validity of a controlled BCH expansion. Although these assumptions are natural for symplectic learned maps close to the exact flow, it remains to clarify when they hold in general infinite-dimensional settings and when they may fail. Future work will address these issues by improving the efficiency of SNO training, extending the framework to broader symplectic phase spaces, and developing sharper theoretical conditions for modified-Hamiltonian stability.

\section{Conclusion}
We proposed the \emph{Symplectic Neural Operator} (SNO), a structure-preserving neural
operator architecture for learning infinite-dimensional Hamiltonian dynamics. The model
preserves symplectic structure by construction, and our theoretical analysis establishes a
modified-Hamiltonian stability result showing that, under suitable Hamiltonian perturbation
assumptions, the learned map nearly conserves a modified Hamiltonian over long rollouts.
Experiments on canonical Hamiltonian PDEs corroborate this theory: SNO achieves improved
long-time stability and energy behavior compared with non-structure-preserving neural
operators while retaining competitive short-time accuracy. These results highlight the
importance of geometric inductive biases in neural operator learning and suggest a principled
path toward reliable structure-preserving surrogate models for complex physical systems.

\section*{Acknowledgement}
Funding in direct support of this work: JST CREST Grant Number 
JPMJCR24Q5,
JSPS KAKENHI Grant Number 25K15148, JST ASPIRE JPMJAP2329, Horizon Europe, MSCA-SE project
101131557 (REMODEL),
Watanuki International Scholarship Foundation.

\bibliographystyle{plainnat}
\bibliography{ref}

\appendix
\section{Proof of the theorems}

\subsection{Proof of Proposition \ref{sno_copm}}
\begin{proof}(Proposition \ref{sno_copm})
Using the characterization $S^*JS=J$,
\[
(ST)^*J(ST)=T^*(S^*JS)T=T^*JT=J,
\]
so $ST$ is symplectic. Also,
\[
J = S^*JS \ \Longrightarrow\ (S^{-1})^*J(S^{-1})=J,
\]
so $S^{-1}$ is symplectic.
\end{proof}

\subsection{Proof of Proposition \ref{prop:linear-shear}}
\begin{proof}(Proposition \ref{prop:linear-shear})
Write $S_A$ in block form on $\bbH_q\oplus\bbH_p$:
\[
S_A = \begin{pmatrix} I & 0 \\ A & I \end{pmatrix},
\qquad
S_A^*=\begin{pmatrix} I & A \\ 0 & I \end{pmatrix}
\quad(\text{since }A^*=A),
\qquad
J=\begin{pmatrix} 0 & -I \\ I & 0 \end{pmatrix}.
\]
Compute
\[
JS_A=\begin{pmatrix} 0 & -I \\ I & 0 \end{pmatrix}\begin{pmatrix} I & 0 \\ A & I \end{pmatrix}
= \begin{pmatrix} -A & -I \\ I & 0 \end{pmatrix},
\]
and then
\[
S_A^*(JS_A)=
\begin{pmatrix} I & A \\ 0 & I \end{pmatrix}
\begin{pmatrix} -A & -I \\ I & 0 \end{pmatrix}
=
\begin{pmatrix} -A + A & -I \\ I & 0 \end{pmatrix}
=
\begin{pmatrix} 0 & -I \\ I & 0 \end{pmatrix}
=J.
\]
Hence $S_A^*JS_A=J$, so $S_A$ is symplectic. The proof for $T_B$ is analogous.
\end{proof}

\subsection{Proof of Proposition \ref{prop:nonlinear-shear}}
\begin{proof}(Proposition \ref{prop:nonlinear-shear})
Since $\omega$ is constant on $\bbP$, it suffices to show that the derivative is symplectic at each point.
Compute the Fr\'echet derivative:
\[
D\Phi_V(q,p)=
\begin{pmatrix}
I & 0\\
D(\nabla V)(q) & I
\end{pmatrix}.
\]
By assumption, $D(\nabla V)(q)$ is self-adjoint, so the calculation in Proposition
\ref{prop:linear-shear} applies pointwise with $A=D(\nabla V)(q)$, giving
\[
(D\Phi_V(q,p))^*\,J\,(D\Phi_V(q,p))=J
\quad\text{for all }(q,p).
\]
Hence $\Phi_V^*\omega=\omega$. The proof for $\Psi_W$ is analogous.
\end{proof}

\subsection{Proof of Theorem \ref{uf}}
\begin{proof}(Theorem \ref{uf})
Let $H:=L^2(\mathbb{T}^n;\mathbb{C}^m)$ and let $\ell^2:=\ell^2(\mathbb{Z}^n;\mathbb{C}^m)$.

\medskip
Step 1 (Parseval implies $\F$ is an isometry).

Write $v=(v_1,\dots,v_m)$ and define $\hat v(k):=(\widehat{v_1}(k),\dots,\widehat{v_m}(k))$,
where $\widehat{v_j}(k)=\int_{\mathbb{T}^n} v_j(x)e^{-2\pi i k\cdot x}\,dx$.
By the Parseval identity on $\mathbb{T}^n$ applied to each component $v_j$ and summing over $j$,
\[
\|v\|_{H}^2
=\sum_{j=1}^m \|v_j\|_{L^2(\mathbb{T}^n)}^2
=\sum_{j=1}^m \sum_{k\in\mathbb{Z}^n} |\widehat{v_j}(k)|^2
=\sum_{k\in\mathbb{Z}^n}\|\hat v(k)\|_{\mathbb{C}^m}^2
=\|\F v\|_{\ell^2}^2.
\]
Hence $\F$ preserves norms: $\|\F v\|_{\ell^2}=\|v\|_{H}$ for all $v\in H$.

More generally, applying Parseval to $v_j,w_j$ componentwise yields
\[
\inner{v}{w}_{H}
=\sum_{j=1}^m \inner{v_j}{w_j}_{L^2}
=\sum_{j=1}^m\sum_{k\in\mathbb{Z}^n}\overline{\widehat{v_j}(k)}\,\widehat{w_j}(k)
=\sum_{k\in\mathbb{Z}^n}\hat v(k)^*\hat w(k)
=\inner{\F v}{\F w}_{\ell^2}.
\]
Thus $\F$ is an inner-product preserving linear map, hence an isometry.

Step 2 (Define $\Finve$ on $\ell^2$ and show $\F\Finve=\mathrm{Id}$).

Take any $\hat v\in \ell^2$. For $N\in\mathbb{N}$ define the partial sums in $H$ by
\[
S_N(x):=\sum_{\substack{k\in\mathbb{Z}^n\\ |k|\le N}}\hat v(k)e^{2\pi i k\cdot x}.
\]
(Here $|k|$ can be any fixed norm on $\mathbb{Z}^n$, e.g. the Euclidean norm.)
Since $\{e_k\}_{k\in\mathbb{Z}^n}$ is an orthonormal basis of $L^2(\mathbb{T}^n)$, Parseval applied componentwise gives, for $N\ge M$,
\[
\|S_N-S_M\|_{H}^2
=\sum_{\substack{k\in\mathbb{Z}^n\\ M<|k|\le N}}\|\hat v(k)\|_{\mathbb{C}^m}^2
\xrightarrow[N,M\to\infty]{}0.
\]
Hence $(S_N)$ is Cauchy in $H$, and since $H$ is complete, there exists $v\in H$ such that
$S_N\to v$ in $H$.
We \emph{define} $\Finve\hat v:=v$; equivalently, $\Finve\hat v$ is the $L^2$-limit of the Fourier series.

Now fix $\ell\in\mathbb{Z}^n$. For $N\ge |\ell|$ we compute
\[
(\F S_N)(\ell)
=\int_{\mathbb{T}^n}\Big(\sum_{|k|\le N}\hat v(k)e^{2\pi i k\cdot x}\Big)e^{-2\pi i \ell\cdot x}\,dx
=\sum_{|k|\le N}\hat v(k)\int_{\mathbb{T}^n}e^{2\pi i (k-\ell)\cdot x}\,dx
=\hat v(\ell),
\]
using orthonormality of $\{e_k\}$ implies orthonormality).
Since $\F$ is bounded (indeed an isometry by Step~1) and $S_N\to v$ in $H$, we have
$\F S_N\to \F v$ in $\ell^2$. Therefore, taking limits in the equality above yields
\[
(\F v)(\ell)=\hat v(\ell)\qquad\text{for all }\ell\in\mathbb{Z}^n,
\]
i.e. $\F(\Finve \hat v)=\hat v$. Hence $\F\Finve=\mathrm{Id}_{\ell^2}$.

Step 3 (Show $\Finve\F=\mathrm{Id}$).
Let $v\in H$ and set $\hat v:=\F v\in \ell^2$.
By definition of $\Finve$ and Parseval, $\Finve \hat v$ is the $L^2$-limit of the Fourier series
with coefficients $\hat v(k)$, which coincides with the expansion of $v$ in the orthonormal basis
$\{e_k\}$ (componentwise). Therefore $\Finve(\F v)=v$, i.e. $\Finve\F=\mathrm{Id}_{H}$.

Hence $\F$ is unitary and $\Finve=\F^*$.
\end{proof}

\subsection{Proof of Theorem \ref{safno}}
\begin{proof}(Theorem \ref{safno})

We prove self-adjointness by computing the adjoint explicitly.

Step 1: Decomposition of the layer operators.

Write
\[
T_k = A_k + B_k,
\]
where
\[
(A_k v)(x) := W_k v(x),
\qquad
B_k := \Finve M_{R_k}\F,
\]
and $M_{R_k}$ is the Fourier multiplier
\[
(M_{R_k}\hat v)(\xi) := R_k(\xi)\hat v(\xi).
\]

Step 2: Boundedness of $T_k$.

Since
\[
\|W_k v(x)\| \le \|W_k\|_{\op}\|v(x)\|
\quad\text{a.e.},
\]
we obtain
\[
\|A_k v\|_{\Hm} \le \|W_k\|_{\op}\|v\|_{\Hm}.
\]
Moreover, using the uniform bound on $R_k$,
\[
\|M_{R_k}\hat v\|_{\ell^2}
\le \Big(\sup_{\xi}\|R_k(\xi)\|_{\op}\Big)\|\hat v\|_{\ell^2}.
\]
Since $\F$ and $\Finve$ are unitary on $\Hm$,
\[
\|B_k\| \le \sup_{\xi}\|R_k(\xi)\|_{\op}.
\]
Hence $T_k$ is bounded.

Step 3: Compute the adjoint of $T_k$.

For $A_k$, we have
\[
\inner{A_k v}{w}_{\Hm}
= \int (W_k v(x))^* w(x)\,dx
= \int v(x)^* W_k^* w(x)\,dx
= \inner{v}{A(W_k^*) w}_{\Hm},
\]
so $A_k^* = A(W_k^*)$.

For $B_k$, since $\Finve=\F^*$,
\[
B_k^* = (\F^* M_{R_k}\F)^* = \F^* M_{R_k}^* \F.
\]
Moreover,
\[
\inner{M_{R_k}\hat v}{\hat w}_{\ell^2}
= \sum_\xi (R_k(\xi)\hat v(\xi))^* \hat w(\xi)
= \sum_\xi \hat v(\xi)^* R_k(\xi)^* \hat w(\xi),
\]
so $M_{R_k}^* = M_{R_k^*}$.
Therefore,
\[
B_k^* = \Finve M_{R_k^*}\F.
\]

Combining,
\[
T_k^* = A(W_k^*) + \Finve M_{R_k^*}\F.
\]

Step 4: Use the adjoint pair symmetry .

By assumption,
\[
W_k^* = W_{L-1-k},
\qquad
R_k(\xi)^* = R_{L-1-k}(\xi),
\]
hence
\[
T_k^* = T_{L-1-k}.
\]

Step 5: Operator representation of $G$.

Define $\mathcal P:\Hd\to\Hm$ by $(\mathcal P a)(x):=Pa(x)$.
Its adjoint satisfies $(\mathcal P)^*=\mathcal P^*$ pointwise.
The recursion yields
\[
G = \mathcal P^*\, T_{L-1}T_{L-2}\cdots T_0\, \mathcal P.
\]

Step 6: Compute $G^*$.

Using $(AB)^*=B^*A^*$ and Step~4,
\[
\begin{aligned}
G^*
&= \mathcal P^* (T_{L-1}\cdots T_0)^* \mathcal P \\
&= \mathcal P^* T_0^* \cdots T_{L-1}^* \mathcal P \\
&= \mathcal P^* T_{L-1}\cdots T_0 \mathcal P \\
&= G.
\end{aligned}
\]

Since $G^*=G$, by definition of adjoint,
\[
\inner{G(a)}{b}_{\Hd}=\inner{a}{G(b)}_{\Hd}
\quad\text{for all }a,b\in\Hd.
\]
Thus $G$ is self-adjoint.
\end{proof}

\subsection{Proof of Theorem \ref{nsafno}}
\begin{proof}(Theorem \ref{nsafno})
Let $\Omega=\mathbb{T}^n$.

Step 1 (A pointwise Lipschitz bound).

Since $\sup_{z}\|D^2\Phi_\theta(z)\|_{\mathrm{op}}\le M$, the gradient is globally Lipschitz:
for all $z_1,z_2\in\mathbb{R}^m$,
\[
\|\nabla\Phi_\theta(z_1)-\nabla\Phi_\theta(z_2)\|
\le M\|z_1-z_2\|.
\]
Also, by the fundamental theorem of calculus along the segment $t\mapsto tz$,
\[
\nabla\Phi_\theta(z)-\nabla\Phi_\theta(0)
=\int_0^1 D^2\Phi_\theta(tz)\,z\,dt,
\]
hence
\[
\|\nabla\Phi_\theta(z)\|\le \|\nabla\Phi_\theta(0)\| + M\|z\|.
\]
Therefore, for any $v\in \Hm$ we have $\nabla\Phi_\theta(v)\in \Hm$ and
\[
\|\nabla\Phi_\theta(v)\|_{\Hm}
\le \|\nabla\Phi_\theta(0)\|\,|\Omega|^{1/2} + M\|v\|_{\Hm},
\]
where $|\Omega|=1$ for Haar probability measure on $\mathbb{T}^n$.

\textbf{Step 2 (Fr\'echet differentiability of $\mathcal{E}_\theta$ and gradient formula).}
Fix $u,h\in\Hd$ and set $v:=Ku\in\Hm$ and $w:=Kh\in\Hm$.
For each $x\in\Omega$, define $\psi_x(t):=\Phi_\theta(v(x)+t w(x))$.
By the one-dimensional fundamental theorem of calculus,
\[
\Phi_\theta(v+w)-\Phi_\theta(v)
=\int_0^1 \nabla\Phi_\theta(v+t w)\cdot w\,dt,
\]
where ``$\cdot$'' is the Euclidean dot product in $\mathbb{R}^m$.
Hence
\[
\mathcal{E}_\theta(u+h)-\mathcal{E}_\theta(u)
=\int_\Omega\int_0^1 \nabla\Phi_\theta(v+t w)\cdot w\,dt\,dx.
\]
Subtract the candidate derivative $\int_\Omega \nabla\Phi_\theta(v)\cdot w\,dx$:
\[
R(u,h)
:=\mathcal{E}_\theta(u+h)-\mathcal{E}_\theta(u)-\int_\Omega \nabla\Phi_\theta(v)\cdot w\,dx
=\int_\Omega\int_0^1\big(\nabla\Phi_\theta(v+t w)-\nabla\Phi_\theta(v)\big)\cdot w\,dt\,dx.
\]
Using the Lipschitz bound from Step 1,
\[
\big|\big(\nabla\Phi_\theta(v+t w)-\nabla\Phi_\theta(v)\big)\cdot w\big|
\le \|\nabla\Phi_\theta(v+t w)-\nabla\Phi_\theta(v)\|\;\|w\|
\le M t \|w\|^2.
\]
Integrating in $t$ and $x$ gives
\[
|R(u,h)|
\le \int_\Omega\int_0^1 M t \|w(x)\|^2\,dt\,dx
=\frac{M}{2}\|w\|_{\Hm}^2
\le \frac{M}{2}\|K\|^2\|h\|_{\Hd}^2.
\]
Therefore $R(u,h)=o(\|h\|_{\Hd})$ as $h\to 0$, and $\mathcal{E}_\theta$ is Fr\'echet differentiable with
\[
D\mathcal{E}_\theta(u)[h]=\int_\Omega \nabla\Phi_\theta((Ku)(x))\cdot (Kh)(x)\,dx
=\inner{\nabla\Phi_\theta(Ku)}{Kh}_{\Hm}.
\]
By definition of the adjoint, $\inner{\nabla\Phi_\theta(Ku)}{Kh}_{\Hm}
=\inner{K^*\nabla\Phi_\theta(Ku)}{h}_{\Hd}$, hence
\[
\nabla\mathcal{E}_\theta(u)=K^*\big(\nabla\Phi_\theta(Ku)\big).
\]

Step 3 (Fr\'echet differentiability of $\mathcal{G}_\theta$ and formula for $D\mathcal{G}_\theta(u)$.

We have $\mathcal{G}_\theta(u)=K^*F(Ku)$ where $F(v):=\nabla\Phi_\theta(v)$ acts pointwise.
Since $K^*$ and $K$ are bounded linear, it suffices to show $F:\Hm\to\Hm$ is Fr\'echet differentiable and compute $DF(v)$.

Fix $v\in\Hm$ and $w\in\Hm$. Pointwise in $\mathbb{R}^m$,
\[
\nabla\Phi_\theta(v+w)-\nabla\Phi_\theta(v) - D^2\Phi_\theta(v)\,w
= \int_0^1 \big(D^2\Phi_\theta(v+t w)-D^2\Phi_\theta(v)\big)\,w\,dt.
\]
Assuming $\Phi_\theta\in C^2$, the map $z\mapsto D^2\Phi_\theta(z)$ is continuous.
Moreover, by boundedness of $D^2\Phi_\theta$,
\[
\|\big(D^2\Phi_\theta(v+t w)-D^2\Phi_\theta(v)\big)\,w\|
\le 2M\|w\|.
\]
Using dominated convergence (in $x$ and $t$) and $w\in L^2$, one obtains
\[
\frac{\|F(v+w)-F(v)-M_{D^2\Phi_\theta(v)}w\|_{\Hm}}{\|w\|_{\Hm}}\to 0
\quad\text{as }\|w\|_{\Hm}\to 0,
\]
so $F$ is Fr\'echet differentiable with
\[
(DF(v)w)(x)=D^2\Phi_\theta(v(x))\,w(x).
\]
Applying the chain rule,
\[
D\mathcal{G}_\theta(u)h
=K^*\big(DF(Ku)(Kh)\big)
=K^*\big(D^2\Phi_\theta(Ku)\,(Kh)\big),
\]
which is the claimed formula.

Step 4 (Boundedness and self-adjointness).

Define the multiplication operator on $\Hm$:
\[
(M_{H_\theta(u)}\varphi)(x):=H_\theta(u)(x)\,\varphi(x),
\quad H_\theta(u)(x)=D^2\Phi_\theta((Ku)(x)).
\]
Since $\|H_\theta(u)(x)\|_{\op}\le M$ for all $x$, we have
\[
\|M_{H_\theta(u)}\varphi\|_{\Hm}\le M\|\varphi\|_{\Hm},
\]
so $M_{H_\theta(u)}$ is bounded.
Thus $D\mathcal{G}_\theta(u)=K^*M_{H_\theta(u)}K$ is bounded.

Finally, for $h_1,h_2\in\Hd$,
\[
\inner{D\mathcal{G}_\theta(u)h_1}{h_2}_{\Hd}
=\inner{M_{H_\theta(u)}Kh_1}{Kh_2}_{\Hm}
=\int_\Omega (Kh_1(x))^\top H_\theta(u)(x)\,(Kh_2(x))\,dx.
\]
Since $H_\theta(u)(x)$ is symmetric for every $x$, the integrand satisfies
\[
(Kh_1)^\top H_\theta(u)\,(Kh_2)=(Kh_2)^\top H_\theta(u)\,(Kh_1),
\]
hence
\[
\inner{D\mathcal{G}_\theta(u)h_1}{h_2}_{\Hd}
=\inner{h_1}{D\mathcal{G}_\theta(u)h_2}_{\Hd}.
\]
Therefore $D\mathcal{G}_\theta(u)$ is self-adjoint.
\end{proof}

\subsection{Proof of Theorem \ref{thm:sno_modified_hamiltonian}}

\begin{lemma}[Hamiltonian commutator identity]\label{lem1}
Let \(F,G:\bbP\to\bbR\) be sufficiently smooth. Then
\[
    [\mathcal{L}_F,\mathcal{L}_G]
    =\mathcal{L}_{\{G,F\}}.
\]
\end{lemma}

\begin{proof}
For any smooth test functional \(A:\bbP\to\bbR\),
\[
    [\mathcal{L}_F,\mathcal{L}_G]A
    =\mathcal{L}_F(\mathcal{L}_G A)-\mathcal{L}_G(\mathcal{L}_F A)
    =\{\{A,G\},F\}-\{\{A,F\},G\}.
\]
By the Jacobi identity for the Poisson bracket,
\[
    \{\{A,G\},F\}-\{\{A,F\},G\}
    =\{A,\{G,F\}\}.
\]
Therefore,
\[
    [\mathcal{L}_F,\mathcal{L}_G]A
    =\mathcal{L}_{\{G,F\}}A.
\]
Since this holds for all test functionals \(A\), the claim follows.
\end{proof}

\begin{lemma}[Modified Hamiltonian generated by BCH]\label{lem2}
Under the BCH assumption, there exists a functional
\[
    \widetilde H_{\varepsilon,m}
    =H+\varepsilon H_1+\varepsilon^2H_2+\cdots+\varepsilon^mH_m
\]
such that
\[
    \exp(\tau\mathcal{L}_{\widetilde H_{\varepsilon,m}})
    =
    \exp(\mathcal{L}_{K_\varepsilon})\exp(\tau\mathcal{L}_H)
    +\mathcal{O}(\varepsilon^{m+1})
\]
as operators acting on smooth functionals on \(\mathcal{K}\). Moreover,
\[
    \norm{\widetilde H_{\varepsilon,m}-H}_{C^0(\mathcal{K})}
    \leq C\varepsilon.
\]
\end{lemma}

\begin{proof}
Let
\[
    A:=\tau\mathcal{L}_H,
    \qquad
    B:=\mathcal{L}_{K_\varepsilon}.
\]
By assumption,
\[
    B=\varepsilon\mathcal{L}_{K_1}+\varepsilon^2\mathcal{L}_{K_2}+\cdots+\varepsilon^m\mathcal{L}_{K_m}+\mathcal{O}(\varepsilon^{m+1}).
\]
The Baker--Campbell--Hausdorff formula gives
\[
\begin{aligned}
    \log(\exp(B)\exp(A))
    &=A+B+\frac12[B,A]
      +\frac{1}{12}[B,[B,A]]
      +\frac{1}{12}[A,[A,B]]
      +\cdots,
\end{aligned}
\]
where only terms up to order \(\varepsilon^m\) are retained.

Because \(A=\tau\mathcal{L}_H\) and each term in \(B\) is a Hamiltonian Lie derivative, Lemma ~\ref{lem1} implies that every nested commutator appearing in the BCH expansion is also a Hamiltonian Lie derivative. Hence there exist functionals \(H_1,\ldots,H_m\) such that
\[
    \log(\exp(B)\exp(A))
    =\tau\mathcal{L}_{H+\varepsilon H_1+\cdots+\varepsilon^mH_m}
     +\mathcal{O}(\varepsilon^{m+1}).
\]
Define
\[
    \widetilde H_{\varepsilon,m}:=H+\varepsilon H_1+\cdots+\varepsilon^mH_m.
\]
Then
\[
    \exp(\mathcal{L}_{K_\varepsilon})\exp(\tau\mathcal{L}_H)
    =\exp(\tau\mathcal{L}_{\widetilde H_{\varepsilon,m}})
     +\mathcal{O}(\varepsilon^{m+1}).
\]
Since the coefficients \(H_j\) are bounded on \(\mathcal{K}\), we also have
\[
    \norm{\widetilde H_{\varepsilon,m}-H}_{C^0(\mathcal{K})}
    \leq C\varepsilon.
\]
\end{proof}

\begin{proof}(Theorem \ref{thm:sno_modified_hamiltonian})

By the Hamiltonian perturbation representation,
\[
    S_\varepsilon=\Phi_{K_\varepsilon}^1\circ\Phi_H^\tau.
\]
For any smooth observable \(F:\bbP\to\bbR\), pullback by \(S_\varepsilon\) gives
\[
    F\circ S_\varepsilon
    =F\circ\Phi_{K_\varepsilon}^1\circ\Phi_H^\tau
    =\exp(\tau\mathcal{L}_H)\exp(\mathcal{L}_{K_\varepsilon})F,
\]
up to the convention of whether Lie operators act on the left or right. Fixing this convention consistently, the Lemma~\ref{lem2} yields a modified Hamiltonian \(\widetilde H_{\varepsilon,m}\) such that the pullback action of \(S_\varepsilon\) agrees with the exact time-\(\tau\) flow of \(\widetilde H_{\varepsilon,m}\) up to order \(\varepsilon^{m+1}\):
\[
    F\circ S_\varepsilon
    =F\circ\Phi_{\widetilde H_{\varepsilon,m}}^\tau
     +\mathcal{O}(\varepsilon^{m+1})
\]
uniformly on \(\mathcal{K}\).

Now choose
\[
    F=\widetilde H_{\varepsilon,m}.
\]
Since the exact Hamiltonian flow of \(\widetilde H_{\varepsilon,m}\) preserves \(\widetilde H_{\varepsilon,m}\),
\[
    \widetilde H_{\varepsilon,m}(\Phi_{\widetilde H_{\varepsilon,m}}^\tau(u))
    =\widetilde H_{\varepsilon,m}(u).
\]
Therefore,
\[
\begin{aligned}
    \widetilde H_{\varepsilon,m}(S_\varepsilon(u))
    &=\widetilde H_{\varepsilon,m}(\Phi_{\widetilde H_{\varepsilon,m}}^\tau(u))
      +\mathcal{O}(\varepsilon^{m+1}) \\
    &=\widetilde H_{\varepsilon,m}(u)+\mathcal{O}(\varepsilon^{m+1}).
\end{aligned}
\]
Thus there exists a constant \(C>0\) such that for all \(u\in\mathcal{K}\),
\[
    \abs{\widetilde H_{\varepsilon,m}(S_\varepsilon(u))-\widetilde H_{\varepsilon,m}(u)}
    \leq C\varepsilon^{m+1}.
\]
Applying this estimate to \(u=u_k\) and summing over \(k=0,\ldots,n-1\), we obtain
\[
\begin{aligned}
    \abs{\widetilde H_{\varepsilon,m}(u_n)-\widetilde H_{\varepsilon,m}(u_0)}
    &\leq
    \sum_{k=0}^{n-1}
    \abs{\widetilde H_{\varepsilon,m}(u_{k+1})-\widetilde H_{\varepsilon,m}(u_k)} \\
    &\leq Cn\varepsilon^{m+1}.
\end{aligned}
\]
This proves the modified-Hamiltonian drift estimate.

Next, since
\[
    \norm{\widetilde H_{\varepsilon,m}-H}_{C^0(\mathcal{K})}\leq C\varepsilon,
\]
we have
\[
\begin{aligned}
    \abs{H(u_n)-H(u_0)}
    &\leq
    \abs{H(u_n)-\widetilde H_{\varepsilon,m}(u_n)} \\
    &\quad+
    \abs{\widetilde H_{\varepsilon,m}(u_n)-\widetilde H_{\varepsilon,m}(u_0)} \\
    &\quad+
    \abs{\widetilde H_{\varepsilon,m}(u_0)-H(u_0)} \\
    &\leq C\varepsilon+Cn\varepsilon^{m+1}+C\varepsilon.
\end{aligned}
\]
Absorbing constants gives
\[
    \abs{H(u_n)-H(u_0)}
    \leq C\varepsilon+Cn\varepsilon^{m+1}.
\]
Finally, if \(n\leq c\varepsilon^{-m}\), then
\[
    n\varepsilon^{m+1}\leq c\varepsilon,
\]
and hence
\[
    \abs{H(u_n)-H(u_0)}\leq C\varepsilon.
\]
\end{proof}

\section{More detail of experiment}
\label{Exper}

\subsection{Data generation}
All datasets used in this paper are generated using the same unified procedure.
The wave equation, Klein--Gordon equation, and electromagnetic wave equation
are treated as instances of a common Hamiltonian evolution problem, and differ
only in the specific form of the governing operator.
We describe the pipeline below using the 1D wave equation as a representative example.

\paragraph{Continuous model.}
We consider Hamiltonian partial differential equations of the general form
\begin{equation}
\partial_t U = \mathcal{J}\,\frac{\delta H}{\delta U},
\qquad
U(\cdot,t)\in\mathbb{V},
\end{equation}
where $\mathbb{V}$ is a function space over a one-dimensional spatial domain,
$\mathcal{J}$ is a skew-adjoint structure operator, and $H$ is the Hamiltonian
functional.
Specific systems (wave, Klein--Gordon, Maxwell) are obtained by specifying
$H$ and $\mathcal{J}$.

\paragraph{Spatial and temporal discretization.}
For all systems, the spatial domain $[0,1]$ is discretized using $N_x$ grid points
\[
x_j = j\Delta x,\qquad j=0,\dots,N_x-1,
\]
and time is advanced using a fixed step size $\Delta t$.
To ensure numerical stability, we enforce a CFL-type condition of the form
\begin{equation}
\frac{c\,\Delta t}{\Delta x}\le 1,
\end{equation}
where $c$ denotes the characteristic wave speed of the system.
Datasets violating this condition are not generated.

\paragraph{Random initial conditions.}
Initial conditions are generated as smooth random functions using truncated
Chebyshev expansions.
Each function is constructed as
\[
\tilde f(x)=\sum_{k=0}^{K} a_k T_k(\xi(x)),
\qquad
\xi(x)=2\frac{x-0}{1-0}-1,
\]
with coefficients $a_k\sim\mathrm{Unif}[-1,1]$.
To enforce homogeneous Dirichlet boundary conditions, we apply the envelope
\begin{equation}
f(x)=\tilde f(x)\,x(1-x),
\end{equation}
which guarantees $f(0)=f(1)=0$.
All initial fields are normalized to lie in $[-1,1]$.
The same procedure is applied independently to each component of the state
(e.g., displacement/velocity or electric/magnetic fields).

\paragraph{Ground-truth time advancement.}
For every system, the continuous PDE is first reduced to a semi-discrete system
of ordinary differential equations via finite differences in space.
Second-order spatial derivatives are approximated using centered difference stencils,
and boundary values are fixed according to the prescribed boundary conditions.
The resulting ODE system is integrated over a single time step $\Delta t$ using
an adaptive Runge--Kutta solver (RK45) with tight tolerances
(rtol=atol=$10^{-8}$).
This defines a one-step ground-truth flow map
\[
U^{n+1}=\Phi_{\Delta t}(U^n).
\]

\paragraph{Dataset construction.}
Each training sample consists of an input--target pair
\[
\texttt{input} = U^n,\qquad
\texttt{target} = U^{n+1},
\]
where $U^n$ and $U^{n+1}$ are represented by their values on the spatial grid.
All datasets differ only in the definition of the underlying Hamiltonian system;
the discretization, randomization, and numerical integration procedures are
identical across experiments.

\paragraph{Dataset size.}
For each physical system, we generate a total of $100{,}000$ independent
input--target pairs.
Each pair corresponds to a randomly sampled initial condition $U^n$ and its
one-step evolution $U^{n+1}=\Phi_{\Delta t}(U^n)$.
All models are trained and evaluated on datasets of identical size generated
using the same pipeline, ensuring a fair comparison across architectures.

\subsection{Model size used in the experiments}

To ensure a fair comparison, all neural operator models were configured to have
comparable model capacity.
Specifically, we matched the number of Fourier modes, hidden channels, and network depth
across SNO, FNO, GNO, and CNO, so that performance differences primarily reflect
architectural structure rather than parameter count.
All models used in the experiments contain approximately $2.3$ million trainable parameters.
For the proposed SNO, this corresponds to an architecture composed of three symplectic
shear-type blocks for the wave and Schr\"odinger equations, six symplectic shear-type
blocks for the electromagnetic equation, and five symplectic shear-type blocks for the
Klein--Gordon equation.

\subsection{Detail about the systems used in the experiments}
\paragraph{Electromagnetic wave equation:}
In the second experiment, we focus on learning and predicting the solutions of an electromagnetic wave equation.
We use the model of the electromagnetic wave equation derived from simplified Maxwell's equations:
\begin{align*}
    \nabla \times \vec{E} &= -\frac{\partial \vec{B}}{\partial t}, \text{(Faraday's law of induction)} \\
    \nabla \times \vec{B} &= \mu_0 \epsilon_0 \frac{\partial \vec{E}}{\partial t}, \text{(Ampère-Maxwell law)}
\end{align*}
The Hamiltonian density of this system is given by 
$
\mathcal{H} = \frac{1}{2}\left(\epsilon_0|\vec{E}|^2+\frac{1}{\mu_0}|\vec{B}|^2\right).
$
We will simplify them to 1D version by choosing 
$\vec{E}=(0,0,\frac{c}{\sqrt{\mu_0}}E)$ where $c = \frac{1}{\sqrt{\mu_0 \epsilon_0}}$ is the wave speed and
$\vec{B}=(0,\frac{1}{\sqrt{\mu_0}}B,0)$. Then, we obtain a Hamiltonian system with the Hamiltonian density $\mathcal{H}=\frac{1}{2}c(E^2+B^2)$. The Hamiltonian equation is given by
$$
\partial_t E = c\,\partial_x B,
\qquad
\partial_t B = c\,\partial_x E.
$$
\paragraph{Schr\"odinger equation:}
In third experiment we focus on learning and predicting the evolution of a quantum wave function
governed by the time-dependent Schr\"odinger equation on the spatial interval $I=[0,1]$.
We assume natural units with $\hbar=1$ and particle mass $m=1$.
The Hamiltonian density of the system is given by
\[
\mathcal{H} = \frac{1}{2}\left(|\partial_x \psi|^2 + V(x)|\psi|^2\right),
\]
where $\psi$ is the complex-valued wave function and $V(x)$ denotes the potential energy.
The time evolution is governed by
\[
i\partial_t \psi(x,t) = -\frac{1}{2}\partial_x^2\psi(x,t) + V(x)\psi(x,t),
\quad x\in[0,1].
\]
We impose Dirichlet boundary conditions,
\(
\psi(0,t)=\psi(1,t)=0,
\)
and prescribe the initial condition $\psi(x,0)=\psi_0(x)$, where $\psi_0$ is a randomly generated
smooth function satisfying the boundary conditions.

\paragraph{Nonlinear Klein--Gordon Equation:}In fourth experiment we consider the nonlinear Klein--Gordon equation on the spatial interval
$I=[0,1]$ with wave speed parameter $c>0$.
This system describes the dynamics of a scalar field $q(x,t)$ with conjugate momentum
$p(x,t)=\partial_t q(x,t)$ and includes both linear and nonlinear potential terms.
The Hamiltonian density is given by
\[
\mathcal{H}
= \frac{1}{2}p^2
+ \frac{1}{2}c^2(\partial_x q)^2
- \frac{1}{2}q^2
- \frac{1}{4}q^4,
\]
which consists of kinetic energy, spatial gradient energy, a quadratic potential,
and a cubic nonlinear interaction.
The corresponding Hamiltonian equations of motion are
\[
\partial_t q(x,t) = \dfrac{\delta \mathcal{H}}{\delta p(x)} = p(x,t), \quad
\partial_t p(x,t) = -\dfrac{\delta \mathcal{H}}{\delta q(x)}
= c^2\partial_x^2 q(x,t) + q(x,t) + q(x,t)^3.
\]

\subsection{More detail about Experimental result}
Since training and test losses reflect only one-step prediction accuracy,
we focus our evaluation on rollout errors and energy conservation,
which are more informative for long-time dynamical behavior.

Quantitative results are summarized from Tables~\ref{tab:sm-avg_errors_time_rows}
to ~\ref{tab:kg-energy_errors_time_rows}.
Across all models, short-time predictions are comparable, while long-time rollouts
reveal substantial differences.
The Symplectic Neural Operator consistently achieves smaller state and energy errors,
indicating improved long-time stability and reduced error accumulation.

\begin{table}[htbp]
\centering
\small
\setlength{\tabcolsep}{6pt}
\begin{tabular}{ccccc}
\hline
Time step & SNO (ours) & FNO & GNO & CNO \\
\hline
1   & \textbf{2.7199e-04 ± 1.3548e-04} & \underline{1.6310e-03 ± 1.4907e-04} & 7.4758e-03 ± 9.4954e-04 & 9.4299e-03 ± 2.0430e-03 \\
10  & \textbf{2.7195e-03 ± 1.3584e-03} & \underline{1.5979e-02 ± 1.5352e-03} & 7.4842e-02 ± 9.4711e-03 & 9.5325e-02 ± 2.2813e-02 \\
20  & \textbf{5.4280e-03 ± 2.7210e-03} & \underline{3.1931e-02 ± 2.9939e-03} & 1.4985e-01 ± 1.8881e-02 & 1.9551e-01 ± 4.9128e-02 \\
30  & \textbf{8.1092e-03 ± 4.0821e-03} & \underline{4.7849e-02 ± 4.3395e-03} & 2.2501e-01 ± 2.8216e-02 & 2.9996e-01 ± 7.7141e-02 \\
40  & \textbf{1.0747e-02 ± 5.4359e-03} & \underline{6.3708e-02 ± 5.6701e-03} & 3.0028e-01 ± 3.7465e-02 & 4.0712e-01 ± 1.0538e-01 \\
50  & \textbf{1.3327e-02 ± 6.7769e-03} & \underline{7.9524e-02 ± 6.9228e-03} & 3.7563e-01 ± 4.6613e-02 & 5.1518e-01 ± 1.3225e-01 \\
60  & \textbf{1.5833e-02 ± 8.0997e-03} & \underline{9.5389e-02 ± 8.1690e-03} & 4.5102e-01 ± 5.5648e-02 & 6.2188e-01 ± 1.5581e-01 \\
70  & \textbf{1.8252e-02 ± 9.3993e-03} & \underline{1.1134e-01 ± 9.3741e-03} & 5.2643e-01 ± 6.4557e-02 & 7.2544e-01 ± 1.7429e-01 \\
80  & \textbf{2.0569e-02 ± 1.0671e-02} & \underline{1.2728e-01 ± 1.0637e-02} & 6.0181e-01 ± 7.3327e-02 & 8.2435e-01 ± 1.8704e-01 \\
90  & \textbf{2.2773e-02 ± 1.1909e-02} & \underline{1.4327e-01 ± 1.1788e-02} & 6.7713e-01 ± 8.1947e-02 & 9.1804e-01 ± 1.9494e-01 \\
100 & \textbf{2.4853e-02 ± 1.3111e-02} & \underline{1.5923e-01 ± 1.2960e-02} & 7.5235e-01 ± 9.0405e-02 & 1.0066e+00 ± 1.9912e-01 \\
200 & \textbf{3.7419e-02 ± 2.2398e-02} & \underline{3.2227e-01 ± 2.3606e-02} & 1.4915e+00 ± 1.6440e-01 & 1.6555e+00 ± 1.7080e-01 \\
300 & \textbf{4.0859e-02 ± 2.3173e-02} & \underline{4.9736e-01 ± 3.7143e-02} & 2.1886e+00 ± 2.1644e-01 & 1.9718e+00 ± 1.4493e-01 \\
400 & \textbf{5.5070e-02 ± 1.9200e-02} & \underline{6.8740e-01 ± 5.7739e-02} & 2.8333e+00 ± 2.4463e-01 & 2.1381e+00 ± 2.0140e-01 \\
500 & \textbf{8.2051e-02 ± 2.4253e-02} & \underline{8.8926e-01 ± 8.4112e-02} & 3.4335e+00 ± 2.5068e-01 & 2.2599e+00 ± 2.6072e-01 \\
600 & \textbf{1.0769e-01 ± 3.9625e-02} & \underline{1.0966e+00 ± 1.1677e-01} & 4.2936e+00 ± 9.8630e-01 & 2.3186e+00 ± 2.6773e-01 \\
700 & \textbf{1.2600e-01 ± 5.6583e-02} & \underline{1.3017e+00 ± 1.5731e-01} & 6.0195e+00 ± 2.1125e+00 & 2.3517e+00 ± 2.5508e-01 \\
800 & \textbf{1.3915e-01 ± 7.1331e-02} & \underline{1.5017e+00 ± 2.0729e-01} & 7.4975e+00 ± 2.6548e+00 & 2.3660e+00 ± 2.3327e-01 \\
900 & \textbf{1.5559e-01 ± 7.4343e-02} & \underline{1.7109e+00 ± 2.8068e-01} & 8.2847e+00 ± 2.9910e+00 & 2.3650e+00 ± 2.0880e-01 \\
1000& \textbf{1.7536e-01 ± 6.6676e-02} & \underline{1.9667e+00 ± 4.6882e-01} & 9.4554e+00 ± 3.3402e+00 & 2.3542e+00 ± 1.8758e-01 \\
\hline
\end{tabular}
\caption{Wave equation: Average $L^2$ error of the state $[f;g]$ at selected rollout steps (mean $\pm$ std over 10 random initial conditions).}
\label{tab:sm-avg_errors_time_rows}
\end{table}

\begin{table}[htbp]
\centering
\small
\setlength{\tabcolsep}{6pt}
\begin{tabular}{ccccc}
\hline
Time step & SNO (ours) & FNO & GNO & CNO \\
\hline
1   & \textbf{1.4459e-06 ± 7.4289e-07} & \underline{2.2608e-04 ± 1.4978e-04} & 9.4928e-04 ± 1.0454e-03 & 2.6988e-03 ± 2.1165e-03 \\
10  & \textbf{1.4953e-05 ± 6.8038e-06} & \underline{2.3125e-03 ± 1.4424e-03} & 1.0321e-02 ± 1.0601e-02 & 2.8130e-02 ± 2.2545e-02 \\
20  & \textbf{3.1735e-05 ± 1.2029e-05} & \underline{4.6844e-03 ± 2.9212e-03} & 2.2827e-02 ± 2.1315e-02 & 5.9586e-02 ± 4.8339e-02 \\
30  & \textbf{5.0383e-05 ± 1.9147e-05} & \underline{7.2688e-03 ± 4.3800e-03} & 3.7477e-02 ± 3.2373e-02 & 9.3951e-02 ± 7.7114e-02 \\
40  & \textbf{7.0818e-05 ± 3.0480e-05} & \underline{1.0082e-02 ± 5.7834e-03} & 5.5042e-02 ± 4.2850e-02 & 1.3015e-01 ± 1.0824e-01 \\
50  & \textbf{9.3056e-05 ± 4.6665e-05} & \underline{1.3204e-02 ± 7.0168e-03} & 7.5484e-02 ± 5.2638e-02 & 1.6695e-01 ± 1.4029e-01 \\
60  & \textbf{1.1687e-04 ± 6.7220e-05} & \underline{1.6456e-02 ± 8.2936e-03} & 9.8500e-02 ± 6.2059e-02 & 2.0661e-01 ± 1.6734e-01 \\
70  & \textbf{1.4221e-04 ± 9.1572e-05} & \underline{2.0054e-02 ± 9.6306e-03} & 1.2403e-01 ± 7.1112e-02 & 2.4542e-01 ± 1.9143e-01 \\
80  & \textbf{1.6884e-04 ± 1.1892e-04} & \underline{2.3902e-02 ± 1.0924e-02} & 1.5203e-01 ± 7.9797e-02 & 2.8278e-01 ± 2.1252e-01 \\
90  & \textbf{1.9918e-04 ± 1.4464e-04} & \underline{2.8001e-02 ± 1.2327e-02} & 1.8242e-01 ± 8.8117e-02 & 3.1881e-01 ± 2.3146e-01 \\
100 & \textbf{2.3046e-04 ± 1.7204e-04} & \underline{3.2277e-02 ± 1.3760e-02} & 2.1514e-01 ± 9.6073e-02 & 3.5376e-01 ± 2.4886e-01 \\
200 & \textbf{5.3384e-04 ± 3.5724e-04} & \underline{8.6136e-02 ± 3.1679e-02} & 6.5480e-01 ± 1.5696e-01 & 6.5382e-01 ± 3.0211e-01 \\
300 & \textbf{6.1899e-04 ± 2.6080e-04} & \underline{1.6373e-01 ± 5.6060e-02} & 1.2476e+00 ± 1.9647e-01 & 8.0728e-01 ± 1.7012e-01 \\
400 & \textbf{5.2980e-04 ± 1.8847e-04} & \underline{2.6476e-01 ± 8.5831e-02} & 1.9339e+00 ± 2.5143e-01 & 8.6420e-01 ± 8.9766e-02 \\
500 & \textbf{5.6292e-04 ± 4.1551e-04} & \underline{3.8614e-01 ± 1.1908e-01} & 2.6946e+00 ± 3.7475e-01 & 9.0783e-01 ± 1.5335e-02 \\
600 & \textbf{1.5332e-03 ± 1.5903e-03} & \underline{5.2914e-01 ± 1.5795e-01} & 5.4733e+00 ± 6.2686e+00 & 9.1476e-01 ± 1.6578e-02 \\
700 & \textbf{2.6907e-03 ± 2.8758e-03} & \underline{6.9558e-01 ± 2.0914e-01} & 1.5929e+01 ± 1.5758e+01 & 9.1541e-01 ± 1.6469e-02 \\
800 & \textbf{2.9956e-03 ± 3.5396e-03} & \underline{8.7450e-01 ± 2.7667e-01} & 1.9374e+01 ± 1.4001e+01 & 9.1565e-01 ± 1.6722e-02 \\
900 & \textbf{2.6197e-03 ± 2.6291e-03} & 1.0742e+00 ± 3.8860e-01 & 2.3433e+01 ± 2.2858e+01 & \underline{9.1561e-01 ± 1.6820e-02} \\
1000& \textbf{2.0607e-03 ± 1.0829e-03} & 1.3534e+00 ± 7.0769e-01 & 3.0558e+01 ± 2.2953e+01 & \underline{9.1559e-01 ± 1.6759e-02} \\
\hline
\end{tabular}
\caption{Wave equation: Average relative energy error at selected rollout steps (mean $\pm$ std over 10 random initial conditions).
Bold indicates the smallest error; underline indicates the second smallest.}
\label{tab:sw-energy_errors_time_rows}
\end{table}

\begin{table}[htbp]
\centering
\small
\setlength{\tabcolsep}{6pt}
\begin{tabular}{ccccc}
\hline
Time step & SNO (ours) & FNO & GNO & CNO \\
\hline
1   & \textbf{4.5632e-04 ± 6.2522e-05} & 1.8360e-02 ± 5.9169e-03 & \underline{8.3330e-04 ± 1.2442e-04} & 4.6406e-03 ± 9.9123e-04 \\
10  & \textbf{4.5745e-03 ± 7.8398e-04} & 1.4679e-01 ± 6.4866e-02 & \underline{8.4677e-03 ± 1.2333e-03} & 4.5723e-02 ± 1.0201e-02 \\
20  & \textbf{9.3278e-03 ± 2.3517e-03} & 2.2794e-01 ± 1.1790e-01 & \underline{1.8027e-02 ± 3.0725e-03} & 9.2067e-02 ± 2.0757e-02 \\
30  & \textbf{1.4660e-02 ± 5.9098e-03} & 3.0493e-01 ± 1.1800e-01 & \underline{3.1311e-02 ± 1.1705e-02} & 1.3905e-01 ± 3.2077e-02 \\
40  & \textbf{2.1306e-02 ± 1.2379e-02} & 3.4633e-01 ± 1.0484e-01 & \underline{5.0656e-02 ± 3.0191e-02} & 1.8647e-01 ± 4.4536e-02 \\
50  & \textbf{3.0452e-02 ± 2.0426e-02} & 3.9122e-01 ± 9.7627e-02 & \underline{7.9810e-02 ± 5.3603e-02} & 2.3440e-01 ± 5.8072e-02 \\
60  & \textbf{4.2399e-02 ± 3.0610e-02} & 4.4391e-01 ± 9.6033e-02 & \underline{1.2264e-01 ± 9.1107e-02} & 2.8286e-01 ± 7.2414e-02 \\
70  & \textbf{5.6973e-02 ± 4.5583e-02} & 4.9884e-01 ± 1.0488e-01 & \underline{1.7742e-01 ± 1.4877e-01} & 3.3184e-01 ± 8.6958e-02 \\
80  & \textbf{7.4384e-02 ± 6.6249e-02} & 5.4432e-01 ± 1.1937e-01 & \underline{2.4194e-01 ± 2.1344e-01} & 3.8121e-01 ± 1.0130e-01 \\
90  & \textbf{9.4769e-02 ± 9.1818e-02} & 6.0122e-01 ± 1.3131e-01 & \underline{3.1552e-01 ± 2.7505e-01} & 4.3118e-01 ± 1.1544e-01 \\
100 & \textbf{1.1766e-01 ± 1.2165e-01} & 6.5749e-01 ± 1.4798e-01 & \underline{3.9436e-01 ± 3.2767e-01} & 4.8184e-01 ± 1.2959e-01 \\
200 & \textbf{4.3124e-01 ± 4.7197e-01} & 1.1975e+00 ± 4.2567e-01 & 1.6674e+00 ± 1.2205e+00 & \underline{1.0066e+00 ± 2.7625e-01} \\
300 & \textbf{1.0474e+00 ± 6.1744e-01} & 1.8022e+00 ± 4.3406e-01 & 4.1721e+00 ± 4.3762e+00 & \underline{1.4704e+00 ± 2.9798e-01} \\
400 & \textbf{1.5647e+00 ± 7.7639e-01} & 2.2607e+00 ± 5.7056e-01 & 6.3104e+00 ± 5.2269e+00 & \underline{1.7706e+00 ± 2.7231e-01} \\
500 & \textbf{1.9844e+00 ± 7.5414e-01} & 2.5713e+00 ± 5.0411e-01 & 1.1378e+01 ± 9.6172e+00 & \underline{1.9626e+00 ± 2.2414e-01} \\
600 & \textbf{2.2745e+00 ± 5.9189e-01} & 2.8325e+00 ± 3.3922e-01 & 1.5642e+01 ± 1.0283e+01 & \underline{2.0881e+00 ± 1.8571e-01} \\
700 & \textbf{2.5724e+00 ± 4.6883e-01} & 3.0952e+00 ± 2.5190e-01 & 2.2671e+01 ± 1.0202e+01 & \underline{2.1514e+00 ± 1.7225e-01} \\
800 & \textbf{2.7710e+00 ± 4.1466e-01} & 3.3068e+00 ± 1.9117e-01 & 2.8407e+01 ± 1.0182e+01 & \underline{2.1758e+00 ± 1.7039e-01} \\
900 & \textbf{2.9575e+00 ± 3.4610e-01} & 3.5580e+00 ± 1.8327e-01 & 3.3353e+01 ± 9.1073e+00 & \underline{2.1990e+00 ± 1.6781e-01} \\
1000& \textbf{3.0579e+00 ± 3.6088e-01} & 3.6944e+00 ± 1.7754e-01 & 3.8659e+01 ± 7.1666e+00 & \underline{2.2088e+00 ± 1.6669e-01} \\
\hline
\end{tabular}
\caption{EM equation: Average $L^2$ error of the state $[E;B]$ at selected rollout steps
(mean $\pm$ std over 10 random initial conditions).
Bold indicates the smallest error; underline indicates the second smallest.}
\label{tab:em-l2_errors_time_rows}
\end{table}

\begin{table}[htbp]
\centering
\small
\setlength{\tabcolsep}{6pt}
\begin{tabular}{ccccc}
\hline
Time step & SNO (ours) & FNO & GNO & CNO \\
\hline
1   & \textbf{9.8087e-05 ± 5.3164e-05} & 1.9989e-03 ± 1.6526e-03 & \underline{5.1449e-04 ± 9.3838e-05} & 2.8742e-03 ± 1.5544e-03 \\
10  & \textbf{9.8581e-04 ± 5.5269e-04} & 1.8536e-02 ± 1.0730e-02 & \underline{5.1631e-03 ± 9.3676e-04} & 2.8716e-02 ± 1.5396e-02 \\
20  & \textbf{2.0142e-03 ± 1.2342e-03} & 3.8661e-02 ± 2.1901e-02 & \underline{1.0355e-02 ± 1.8497e-03} & 5.7328e-02 ± 3.0679e-02 \\
30  & \textbf{3.1119e-03 ± 2.1417e-03} & 4.5181e-02 ± 2.6999e-02 & \underline{1.5465e-02 ± 2.7265e-03} & 8.5912e-02 ± 4.6158e-02 \\
40  & \textbf{4.2724e-03 ± 3.2747e-03} & 5.6596e-02 ± 3.2607e-02 & \underline{2.0352e-02 ± 3.9589e-03} & 1.1446e-01 ± 6.1969e-02 \\
50  & \textbf{5.4486e-03 ± 4.5561e-03} & 6.6241e-02 ± 4.0407e-02 & \underline{2.4781e-02 ± 6.2967e-03} & 1.4310e-01 ± 7.7985e-02 \\
60  & \textbf{6.6052e-03 ± 5.9261e-03} & 7.1880e-02 ± 4.7052e-02 & \underline{2.8857e-02 ± 9.9159e-03} & 1.7202e-01 ± 9.4025e-02 \\
70  & \textbf{7.7903e-03 ± 7.3125e-03} & 7.6019e-02 ± 5.2680e-02 & \underline{3.3675e-02 ± 1.2497e-02} & 2.0138e-01 ± 1.0974e-01 \\
80  & \textbf{9.3647e-03 ± 8.2584e-03} & 8.5042e-02 ± 5.7537e-02 & \underline{4.0258e-02 ± 1.2403e-02} & 2.3117e-01 ± 1.2501e-01 \\
90  & \textbf{1.0967e-02 ± 9.0265e-03} & 9.3285e-02 ± 6.3119e-02 & \underline{4.8868e-02 ± 1.1959e-02} & 2.6140e-01 ± 1.3965e-01 \\
100 & \textbf{1.2541e-02 ± 9.5424e-03} & 1.0140e-01 ± 7.0104e-02 & \underline{5.8771e-02 ± 1.5554e-02} & 2.9200e-01 ± 1.5347e-01 \\
200 & \textbf{4.0546e-02 ± 5.2670e-02} & \underline{1.2702e-01 ± 9.9095e-02} & 5.2447e-01 ± 7.3528e-01 & 5.7894e-01 ± 2.1195e-01 \\
300 & \textbf{7.4135e-02 ± 4.6023e-02} & \underline{7.7263e-02 ± 6.5141e-02} & 4.9324e+00 ± 1.2466e+01 & 7.6135e-01 ± 1.5741e-01 \\
400 & \textbf{9.8622e-02 ± 9.8602e-02} & \underline{1.5892e-01 ± 1.0879e-01} & 9.3367e+00 ± 1.7393e+01 & 8.6537e-01 ± 1.1184e-01 \\
500 & \textbf{1.0531e-01 ± 7.0550e-02} & \underline{2.5009e-01 ± 1.4615e-01} & 3.7433e+01 ± 5.1763e+01 & 9.1735e-01 ± 8.4738e-02 \\
600 & \textbf{9.8902e-02 ± 6.7705e-02} & \underline{3.3596e-01 ± 1.6195e-01} & 6.3051e+01 ± 6.7667e+01 & 9.3905e-01 ± 5.4070e-02 \\
700 & \textbf{8.6077e-02 ± 7.8624e-02} & \underline{4.5346e-01 ± 2.1413e-01} & 1.2005e+02 ± 8.4868e+01 & 9.4772e-01 ± 2.4651e-02 \\
800 & \textbf{1.1433e-01 ± 7.9819e-02} & \underline{5.8063e-01 ± 2.7695e-01} & 1.8029e+02 ± 9.5279e+01 & 9.5326e-01 ± 8.7271e-03 \\
900 & \textbf{8.5777e-02 ± 6.6255e-02} & \underline{7.2418e-01 ± 3.2390e-01} & 2.3878e+02 ± 9.7266e+01 & 9.5539e-01 ± 7.6726e-03 \\
1000& \textbf{7.6348e-02 ± 5.8800e-02} & \underline{8.8634e-01 ± 3.5964e-01} & 3.1605e+02 ± 9.9995e+01 & 9.5688e-01 ± 6.7310e-03 \\
\hline
\end{tabular}
\caption{EM equation: Average relative energy error at selected rollout steps
(mean $\pm$ std over 10 random initial conditions).
Bold indicates the smallest error; underline indicates the second smallest.}
\label{tab:em-energy_errors_time_rows}
\end{table}

\begin{table}[htbp]
\centering
\small
\setlength{\tabcolsep}{6pt}
\begin{tabular}{ccccc}
\hline
Time step & SNO (ours) & FNO & GNO & CNO \\
\hline
1   & \textbf{8.5307e-04 ± 4.1865e-05} & \underline{2.1574e-03 ± 1.2118e-04} & 2.3512e-02 ± 7.5493e-04 & 2.0376e-02 ± 2.4352e-03 \\
10  & \textbf{3.4027e-02 ± 1.1334e-02} & \underline{3.7716e-02 ± 1.0094e-02} & 2.3493e-01 ± 8.5029e-03 & 1.8306e-01 ± 2.3398e-02 \\
20  & \textbf{5.8060e-02 ± 1.7743e-02} & \underline{7.0443e-02 ± 1.6704e-02} & 4.8192e-01 ± 1.7654e-02 & 3.2496e-01 ± 4.2910e-02 \\
30  & \textbf{7.7376e-02 ± 2.2805e-02} & \underline{1.0946e-01 ± 2.8500e-02} & 7.4525e-01 ± 2.9251e-02 & 4.3305e-01 ± 5.9118e-02 \\
40  & \textbf{9.0340e-02 ± 2.5473e-02} & \underline{1.8820e-01 ± 4.7966e-02} & 1.0279e+00 ± 4.3341e-02 & 5.1313e-01 ± 7.0335e-02 \\
50  & \textbf{1.0123e-01 ± 2.6573e-02} & \underline{2.5977e-01 ± 4.4975e-02} & 1.3332e+00 ± 6.0060e-02 & 5.7321e-01 ± 7.8355e-02 \\
60  & \textbf{1.1056e-01 ± 2.8871e-02} & \underline{3.1678e-01 ± 2.6782e-02} & 1.6646e+00 ± 7.9882e-02 & 6.1843e-01 ± 8.2276e-02 \\
70  & \textbf{1.1949e-01 ± 2.9432e-02} & \underline{3.5741e-01 ± 1.9874e-02} & 2.0280e+00 ± 1.0395e-01 & 6.5249e-01 ± 8.1985e-02 \\
80  & \textbf{1.2549e-01 ± 2.8100e-02} & \underline{3.6321e-01 ± 1.5993e-02} & 2.4307e+00 ± 1.3449e-01 & 6.7503e-01 ± 8.2507e-02 \\
90  & \textbf{1.3202e-01 ± 2.5894e-02} & \underline{3.7548e-01 ± 1.1533e-02} & 2.8825e+00 ± 1.7327e-01 & 6.9118e-01 ± 8.6357e-02 \\
100 & \textbf{1.3949e-01 ± 2.4032e-02} & \underline{3.9677e-01 ± 1.3261e-02} & 3.3966e+00 ± 2.2466e-01 & 7.0412e-01 ± 8.8933e-02 \\
200 & \textbf{2.5536e-01 ± 1.8225e-02} & \underline{5.2072e-01 ± 1.3291e-02} & 1.7747e+01 ± 4.0772e-01 & 7.4426e-01 ± 9.2955e-02 \\
300 & \textbf{3.6039e-01 ± 2.1566e-02} & \underline{6.8729e-01 ± 1.9123e-02} & 1.1920e+01 ± 2.4516e+00 & 7.4639e-01 ± 7.8916e-02 \\
400 & \textbf{4.4901e-01 ± 2.3022e-02} & 8.4749e-01 ± 5.2231e-02 & 1.3544e+01 ± 3.2344e+00 & \underline{7.4160e-01 ± 7.7576e-02} \\
500 & \textbf{5.3155e-01 ± 2.2694e-02} & 1.0712e+00 ± 9.0086e-02 & 1.6169e+01 ± 2.9063e+00 & \underline{7.4304e-01 ± 8.0398e-02} \\
600 & \textbf{6.0207e-01 ± 3.0264e-02} & 1.3653e+00 ± 1.4241e-01 & 1.7702e+01 ± 1.0087e+00 & \underline{7.5304e-01 ± 8.2394e-02} \\
700 & \textbf{6.8163e-01 ± 9.0493e-02} & 1.7077e+00 ± 1.7906e-01 & 1.8326e+01 ± 1.4340e+00 & \underline{7.5534e-01 ± 1.0624e-01} \\
800 & \textbf{7.8170e-01 ± 1.7954e-01} & 2.0900e+00 ± 2.0920e-01 & 1.7047e+01 ± 1.6237e+00 & \underline{7.4949e-01 ± 1.0087e-01} \\
900 & \textbf{9.4456e-01 ± 3.4879e-01} & 2.5013e+00 ± 2.4007e-01 & 1.8201e+01 ± 1.4458e+00 & \underline{7.3382e-01 ± 1.0244e-01} \\
1000& \underline{2.0293e+00 ± 3.2627e+00} & 2.9413e+00 ± 2.7231e-01 & 1.7872e+01 ± 1.1777e+00 & \textbf{7.2473e-01 ± 1.0187e-01} \\
\hline
\end{tabular}
\caption{Schr\"odinger equation: Average $L^2$ error of the state $[re;im]$ at selected rollout steps
(mean $\pm$ std over 10 random initial conditions).
Bold indicates the smallest error; underline indicates the second smallest.}
\label{tab:sch-l2_errors_time_rows}
\end{table}

\begin{table}[htbp]
\centering
\small
\setlength{\tabcolsep}{6pt}
\begin{tabular}{ccccc}
\hline
Time step & SNO (ours) & FNO & GNO & CNO \\
\hline
1   & \textbf{2.2896e-04 ± 2.1409e-04} & \underline{3.5034e-04 ± 1.7306e-04} & 1.0287e-02 ± 7.1902e-03 & 2.8503e-02 ± 4.6709e-03 \\
10  & \underline{1.7667e-02 ± 9.0209e-03} & \textbf{1.0391e-02 ± 7.6222e-03} & 8.9712e-02 ± 3.2549e-02 & 2.6761e-01 ± 2.1914e-02 \\
20  & \textbf{4.4489e-02 ± 2.1620e-02} & \underline{5.9262e-02 ± 5.9274e-02} & 2.0089e-01 ± 7.5031e-02 & 4.8165e-01 ± 2.5188e-02 \\
30  & \textbf{5.4286e-02 ± 2.7906e-02} & \underline{1.6843e-01 ± 1.8578e-01} & 3.6150e-01 ± 1.1656e-01 & 6.3685e-01 ± 2.3837e-02 \\
40  & \textbf{6.4247e-02 ± 2.9048e-02} & 7.7911e-01 ± 5.9917e-01 & \underline{5.0758e-01 ± 1.3124e-01} & 7.3193e-01 ± 1.9284e-02 \\
50  & \textbf{7.0698e-02 ± 2.2359e-02} & 1.5130e+00 ± 8.3417e-01 & \underline{7.4457e-01 ± 2.0216e-01} & 8.0375e-01 ± 2.0818e-02 \\
60  & \textbf{5.8614e-02 ± 3.2303e-02} & 2.0131e+00 ± 4.3464e-01 & 1.0311e+00 ± 2.3215e-01 & \underline{8.5674e-01 ± 1.4883e-02} \\
70  & \textbf{5.6767e-02 ± 3.0466e-02} & 3.4641e+00 ± 9.3056e-01 & 1.3442e+00 ± 2.3951e-01 & \underline{8.9093e-01 ± 9.9771e-03} \\
80  & \textbf{4.4155e-02 ± 2.3716e-02} & 2.0692e+00 ± 5.0638e-01 & 1.5952e+00 ± 2.9500e-01 & \underline{9.0625e-01 ± 1.4425e-02} \\
90  & \textbf{6.2690e-02 ± 2.9814e-02} & 3.3299e+00 ± 8.0265e-01 & 2.0815e+00 ± 4.5458e-01 & \underline{9.2260e-01 ± 1.0374e-02} \\
100 & \textbf{5.2798e-02 ± 2.8234e-02} & 2.3427e+00 ± 4.8016e-01 & 2.6017e+00 ± 4.7998e-01 & \underline{9.3565e-01 ± 6.4499e-03} \\
200 & \textbf{7.5000e-02 ± 3.2158e-02} & 3.2773e+00 ± 7.5053e-01 & 6.3301e+01 ± 2.1092e+01 & \underline{9.6816e-01 ± 5.3943e-03} \\
300 & \textbf{7.5988e-02 ± 4.3304e-02} & 4.3426e+00 ± 1.2464e+00 & 2.3931e+02 ± 1.4370e+02 & \underline{9.6994e-01 ± 5.8477e-03} \\
400 & \textbf{9.3727e-02 ± 4.6479e-02} & 4.5848e+00 ± 1.0946e+00 & 2.4662e+02 ± 1.1122e+02 & \underline{9.6948e-01 ± 5.9812e-03} \\
500 & \textbf{1.3088e-01 ± 6.5064e-02} & 3.6530e+00 ± 9.3833e-01 & 2.4690e+02 ± 9.7077e+01 & \underline{9.6968e-01 ± 5.9974e-03} \\
600 & \textbf{1.2366e-01 ± 5.5656e-02} & 4.8098e+00 ± 1.2750e+00 & 2.4635e+02 ± 9.4433e+01 & \underline{9.6961e-01 ± 5.8986e-03} \\
700 & \textbf{1.7737e-01 ± 1.8253e-01} & 5.7553e+00 ± 1.2348e+00 & 3.3578e+02 ± 1.6508e+02 & \underline{9.6980e-01 ± 5.6918e-03} \\
800 & \textbf{2.0660e-01 ± 2.3856e-01} & 4.6042e+00 ± 1.2803e+00 & 2.8190e+02 ± 7.5350e+01 & \underline{9.6969e-01 ± 5.7903e-03} \\
900 & \underline{1.7104e+00 ± 4.8311e+00} & 5.1049e+00 ± 9.8926e-01 & 2.9552e+02 ± 1.5569e+02 & \textbf{9.6981e-01 ± 5.6817e-03} \\
1000& 1.7042e+02 ± 5.3805e+02 & \underline{6.3522e+00 ± 1.2201e+00} & 3.3652e+02 ± 1.8009e+02 & \textbf{9.6957e-01 ± 5.9932e-03} \\
\hline
\end{tabular}
\caption{Schr\"odinger equation: Average relative energy error at selected rollout steps
(mean $\pm$ std over 10 random initial conditions).
Bold indicates the smallest error; underline indicates the second smallest.}
\label{tab:sch-energy_errors_time_rows}
\end{table}

\begin{table}[htbp]
\centering
\small
\setlength{\tabcolsep}{6pt}
\begin{tabular}{ccccc}
\hline
Time step & SNO (ours) & FNO & GNO & CNO \\
\hline
1   & \textbf{5.1607e-04 ± 1.2732e-05} & \underline{9.0181e-04 ± 6.1741e-05} & 1.0153e-03 ± 1.0656e-04 & 1.4576e-02 ± 2.0963e-03 \\
10  & \textbf{5.1710e-03 ± 1.3257e-04} & \underline{9.0388e-03 ± 6.3539e-04} & 1.0221e-02 ± 1.0636e-03 & 1.4649e-01 ± 1.9070e-02 \\
20  & \textbf{1.0383e-02 ± 2.7185e-04} & \underline{1.8190e-02 ± 1.3207e-03} & 2.0756e-02 ± 2.1446e-03 & 3.0649e-01 ± 3.2829e-02 \\
30  & \textbf{1.5633e-02 ± 4.2165e-04} & \underline{2.7504e-02 ± 2.0590e-03} & 3.1843e-02 ± 3.2972e-03 & 4.8654e-01 ± 4.4339e-02 \\
40  & \textbf{2.0928e-02 ± 5.7906e-04} & \underline{3.7010e-02 ± 2.8402e-03} & 4.3687e-02 ± 4.6007e-03 & 6.8317e-01 ± 5.9875e-02 \\
50  & \textbf{2.6274e-02 ± 7.4180e-04} & \underline{4.6739e-02 ± 3.6385e-03} & 5.6438e-02 ± 6.1535e-03 & 8.8811e-01 ± 8.1802e-02 \\
60  & \textbf{3.1687e-02 ± 9.0872e-04} & \underline{5.6733e-02 ± 4.4417e-03} & 7.0185e-02 ± 8.0584e-03 & 1.0940e+00 ± 1.0806e-01 \\
70  & \textbf{3.7177e-02 ± 1.0814e-03} & \underline{6.7065e-02 ± 5.2563e-03} & 8.4952e-02 ± 1.0402e-02 & 1.2961e+00 ± 1.3272e-01 \\
80  & \textbf{4.2754e-02 ± 1.2621e-03} & \underline{7.7821e-02 ± 6.1272e-03} & 1.0070e-01 ± 1.3238e-02 & 1.4921e+00 ± 1.5080e-01 \\
90  & \textbf{4.8434e-02 ± 1.4513e-03} & \underline{8.9071e-02 ± 7.1606e-03} & 1.1735e-01 ± 1.6575e-02 & 1.6822e+00 ± 1.6160e-01 \\
100 & \textbf{5.4225e-02 ± 1.6513e-03} & \underline{1.0085e-01 ± 8.4698e-03} & 1.3477e-01 ± 2.0383e-02 & 1.8665e+00 ± 1.6816e-01 \\
200 & \textbf{1.1745e-01 ± 5.2180e-03} & \underline{2.4736e-01 ± 4.6000e-02} & 3.1563e-01 ± 6.3976e-02 & 3.1312e+00 ± 2.9415e-01 \\
300 & \textbf{1.8819e-01 ± 9.4619e-03} & \underline{4.0734e-01 ± 9.5143e-02} & 4.6690e-01 ± 7.7026e-02 & 3.6836e+00 ± 4.3960e-01 \\
400 & \textbf{2.7563e-01 ± 1.9537e-02} & 7.4179e-01 ± 7.0877e-01 & \underline{6.3418e-01 ± 8.4716e-02} & 4.0809e+00 ± 4.3408e-01 \\
500 & \textbf{3.6969e-01 ± 2.7930e-02} & 1.3786e+00 ± 2.3600e+00 & \underline{8.2240e-01 ± 7.7240e-02} & 4.5524e+00 ± 4.6611e-01 \\
600 & \textbf{4.7114e-01 ± 3.5487e-02} & 1.7371e+00 ± 3.0915e+00 & \underline{1.0134e+00 ± 9.3464e-02} & 5.1024e+00 ± 6.9532e-01 \\
700 & \textbf{5.7598e-01 ± 4.1200e-02} & 2.4182e+00 ± 4.9366e+00 & \underline{2.0644e+00 ± 2.6931e+00} & 5.6771e+00 ± 9.9939e-01 \\
800 & \textbf{7.0047e-01 ± 5.7071e-02} & \underline{3.0065e+00 ± 6.3882e+00} & 4.4241e+00 ± 4.1572e+00 & 6.2705e+00 ± 1.2652e+00 \\
900 & \textbf{8.3745e-01 ± 7.4744e-02} & \underline{3.3427e+00 ± 7.1788e+00} & 7.4576e+00 ± 4.2263e+00 & 6.8423e+00 ± 1.4684e+00 \\
1000& \textbf{9.6130e-01 ± 7.0418e-02} & \underline{3.8190e+00 ± 8.3909e+00} & 1.0321e+01 ± 3.3743e+00 & 7.3353e+00 ± 1.5668e+00 \\
\hline
\end{tabular}
\caption{Klein--Gordon equation: Average $L^2$ error of the state $[f;g]$ at selected rollout steps
(mean $\pm$ std over 10 random initial conditions).
Bold indicates the smallest error; underline indicates the second smallest.}
\label{tab:kg-l2_errors_time_rows}
\end{table}

\begin{table}[htbp]
\centering
\small
\setlength{\tabcolsep}{6pt}
\begin{tabular}{ccccc}
\hline
Time step & SNO (ours) & FNO & GNO & CNO \\
\hline
1   & \textbf{2.8794e-05 ± 2.5831e-05} & 1.6809e-04 ± 1.3114e-04 & \underline{1.5536e-04 ± 1.1380e-04} & 5.7628e-03 ± 3.9218e-03 \\
10  & \textbf{3.0437e-04 ± 2.5164e-04} & 1.7006e-03 ± 1.3030e-03 & \underline{1.4132e-03 ± 1.1082e-03} & 7.1123e-02 ± 4.7740e-02 \\
20  & \textbf{6.4672e-04 ± 5.0263e-04} & 3.3962e-03 ± 2.5427e-03 & \underline{2.5020e-03 ± 2.1754e-03} & 1.7335e-01 ± 1.2483e-01 \\
30  & \textbf{1.0353e-03 ± 7.6369e-04} & 5.1070e-03 ± 3.6807e-03 & \underline{3.4219e-03 ± 3.0486e-03} & 3.1005e-01 ± 2.3616e-01 \\
40  & \textbf{1.4957e-03 ± 1.0085e-03} & 6.8440e-03 ± 4.7033e-03 & \underline{4.4229e-03 ± 3.5113e-03} & 4.8235e-01 ± 3.6694e-01 \\
50  & \textbf{2.0341e-03 ± 1.2330e-03} & 8.5923e-03 ± 5.6330e-03 & \underline{5.5079e-03 ± 3.6509e-03} & 6.8451e-01 ± 4.8825e-01 \\
60  & \textbf{2.6323e-03 ± 1.4682e-03} & 1.0330e-02 ± 6.5026e-03 & \underline{6.5765e-03 ± 3.8032e-03} & 8.9162e-01 ± 6.0118e-01 \\
70  & \textbf{3.2861e-03 ± 1.7253e-03} & 1.2041e-02 ± 7.3328e-03 & \underline{7.6205e-03 ± 4.2832e-03} & 1.1010e+00 ± 6.9185e-01 \\
80  & \textbf{3.9902e-03 ± 2.0105e-03} & 1.3704e-02 ± 8.1259e-03 & \underline{8.6280e-03 ± 5.3431e-03} & 1.3145e+00 ± 7.5504e-01 \\
90  & \textbf{4.7392e-03 ± 2.3264e-03} & 1.5299e-02 ± 8.8821e-03 & \underline{9.6698e-03 ± 6.8988e-03} & 1.5315e+00 ± 8.0374e-01 \\
100 & \textbf{5.5267e-03 ± 2.6738e-03} & 1.6823e-02 ± 9.6004e-03 & \underline{1.1395e-02 ± 7.9654e-03} & 1.7449e+00 ± 8.4807e-01 \\
200 & \textbf{1.4242e-02 ± 6.9533e-03} & \underline{2.9454e-02 ± 1.7624e-02} & 3.7614e-02 ± 2.6844e-02 & 2.8202e+00 ± 1.0713e+00 \\
300 & \textbf{2.1929e-02 ± 8.3854e-03} & \underline{4.1483e-02 ± 3.2338e-02} & 7.2169e-02 ± 3.4758e-02 & 2.8065e+00 ± 1.0585e+00 \\
400 & \textbf{2.5658e-02 ± 9.5577e-03} & 1.4474e-01 ± 3.0159e-01 & \underline{9.6307e-02 ± 3.3169e-02} & 2.6558e+00 ± 9.4671e-01 \\
500 & \textbf{2.4713e-02 ± 1.2180e-02} & 7.3091e-01 ± 2.1192e+00 & \underline{9.5030e-02 ± 3.6458e-02} & 2.4979e+00 ± 9.3078e-01 \\
600 & \textbf{2.0020e-02 ± 1.4046e-02} & 1.2669e+00 ± 3.7711e+00 & \underline{6.8152e-02 ± 4.2556e-02} & 2.3098e+00 ± 1.0045e+00 \\
700 & \textbf{1.9395e-02 ± 1.4944e-02} & 2.6338e+00 ± 8.0354e+00 & \underline{4.9562e-01 ± 1.3850e+00} & 2.1644e+00 ± 1.0892e+00 \\
800 & \textbf{2.1471e-02 ± 1.6639e-02} & 4.1927e+00 ± 1.2876e+01 & \underline{1.9710e+00 ± 2.5730e+00} & 2.0354e+00 ± 1.1428e+00 \\
900 & \textbf{2.5336e-02 ± 1.8135e-02} & 5.3221e+00 ± 1.6352e+01 & 4.8306e+00 ± 3.9745e+00 & \underline{1.8875e+00 ± 1.0700e+00} \\
1000& \textbf{3.0421e-02 ± 1.9090e-02} & 6.7121e+00 ± 2.0630e+01 & 8.8947e+00 ± 4.9985e+00 & \underline{1.7986e+00 ± 9.9847e-01} \\
\hline
\end{tabular}
\caption{Klein--Gordon equation: Average relative energy error at selected rollout steps
(mean $\pm$ std over 10 random initial conditions).
Bold indicates the smallest error; underline indicates the second smallest.}
\label{tab:kg-energy_errors_time_rows}
\end{table}

\begin{figure}[htbp]
  \centering
  \includegraphics[width=0.6\linewidth]{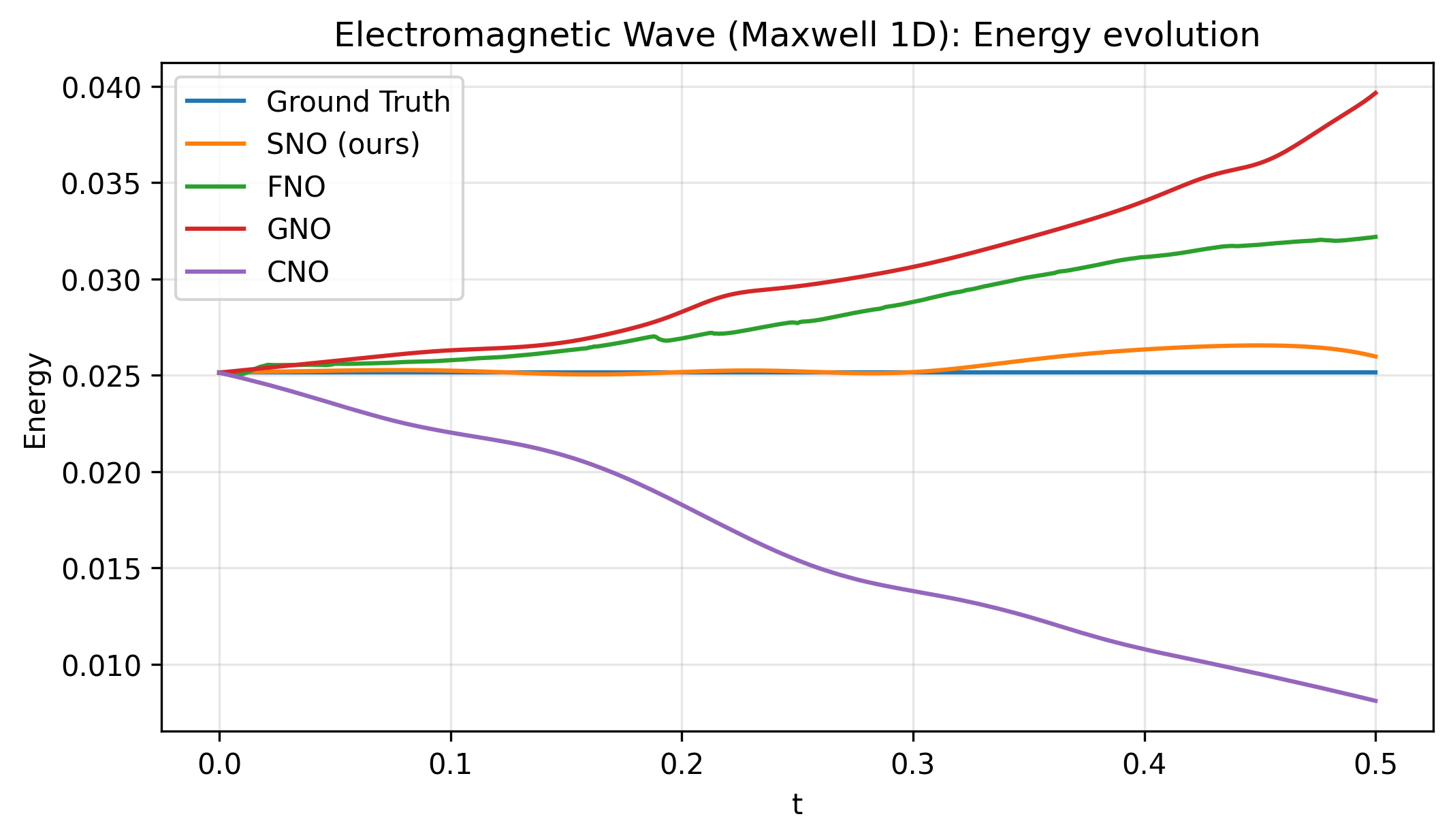}
  \caption{Maxwell 1D equation: Energy evolution.}
  \label{fig:em-energy}
\end{figure}

\begin{figure}[htbp]
  \centering
  \includegraphics[width=\linewidth]{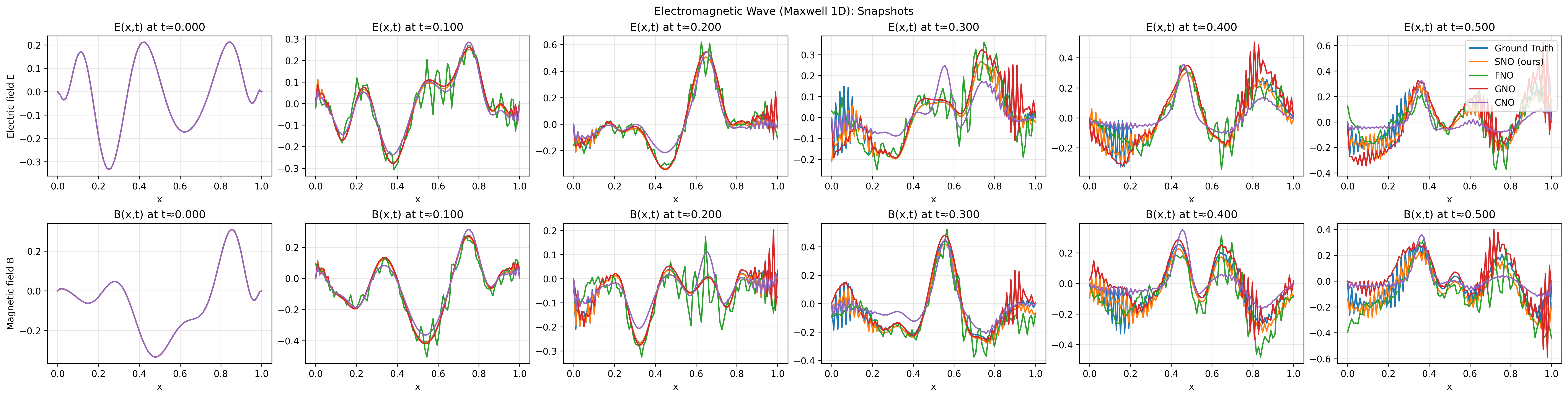}
  \caption{Maxwell 1D equation: snapshots of electromagnetic field at first 500 time-steps}
  \label{fig:em-snapshots}
\end{figure}

\begin{figure}[htbp]
  \centering
  \includegraphics[width=\linewidth]{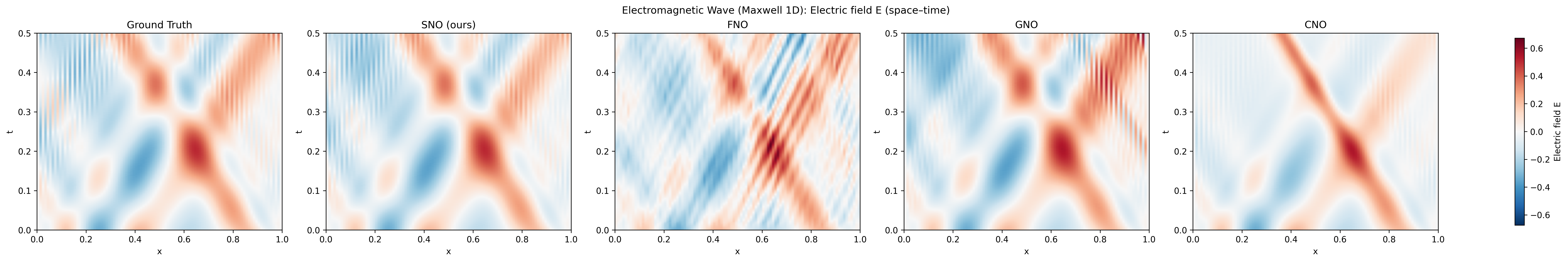}
  \caption{Maxwell 1D equation: space--time Electric field}
  \label{fig:em-spacetime-E}
\end{figure}

\begin{figure}[htbp]
  \centering
  \includegraphics[width=\linewidth]{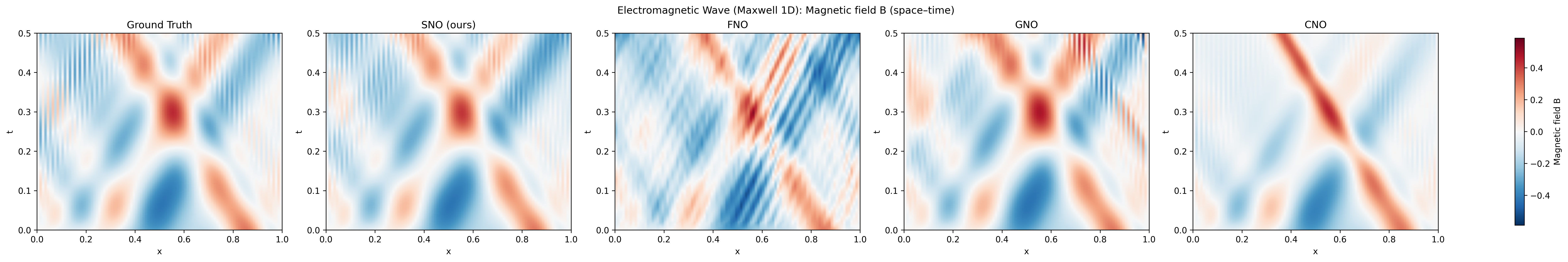}
  \caption{Maxwell 1D equation: space--time Magnetic field}
  \label{fig:em-spacetime-B}
\end{figure}

\begin{figure}[htbp]
  \centering
  \includegraphics[width=\linewidth]{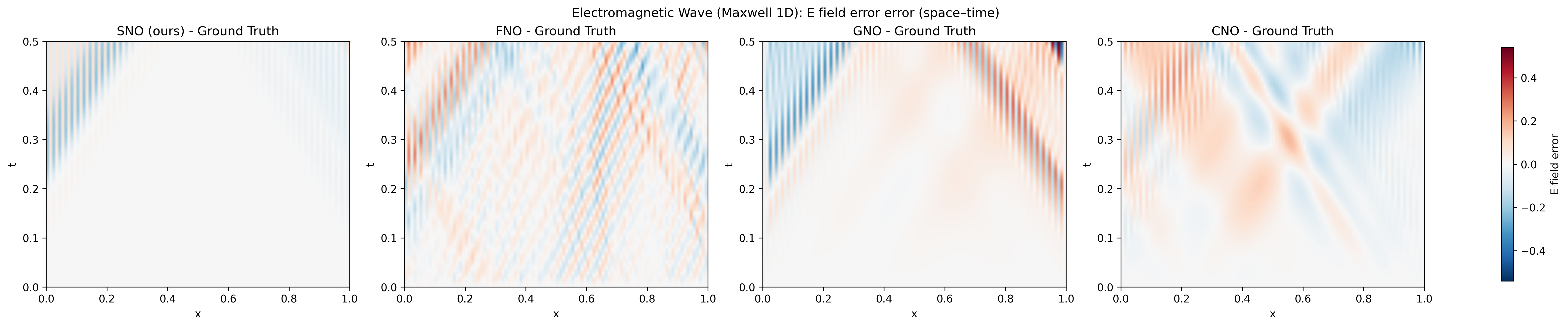}
  \caption{Maxwell 1D equation: space--time electric field error}
  \label{fig:em-spacetime-error-E}
\end{figure}
\begin{figure}[htbp]
  \centering
  \includegraphics[width=\linewidth]{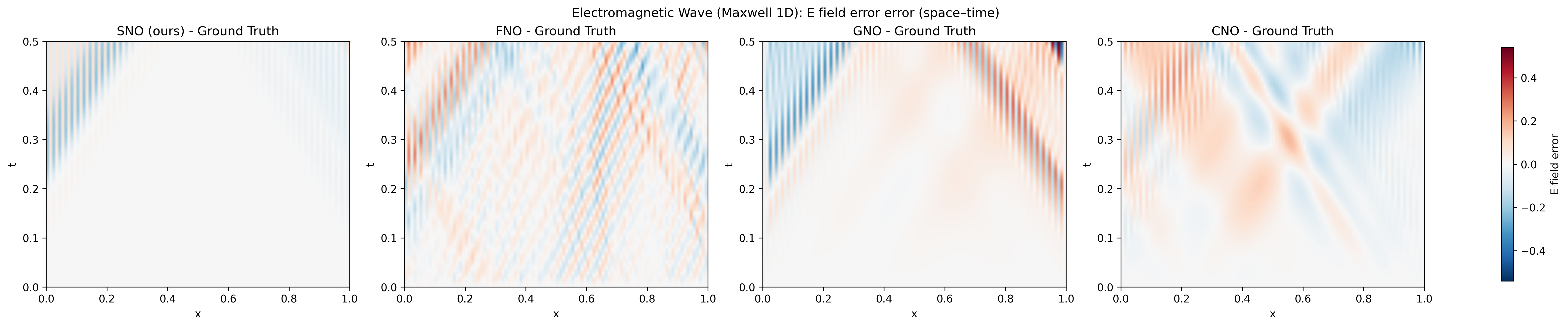}
  \caption{Maxwell 1D equation: space--time Magnetic field error
  }
  \label{fig:em-spacetime-error-B}
\end{figure}

\begin{figure}[htbp]
  \centering
  \includegraphics[width=0.6\linewidth]{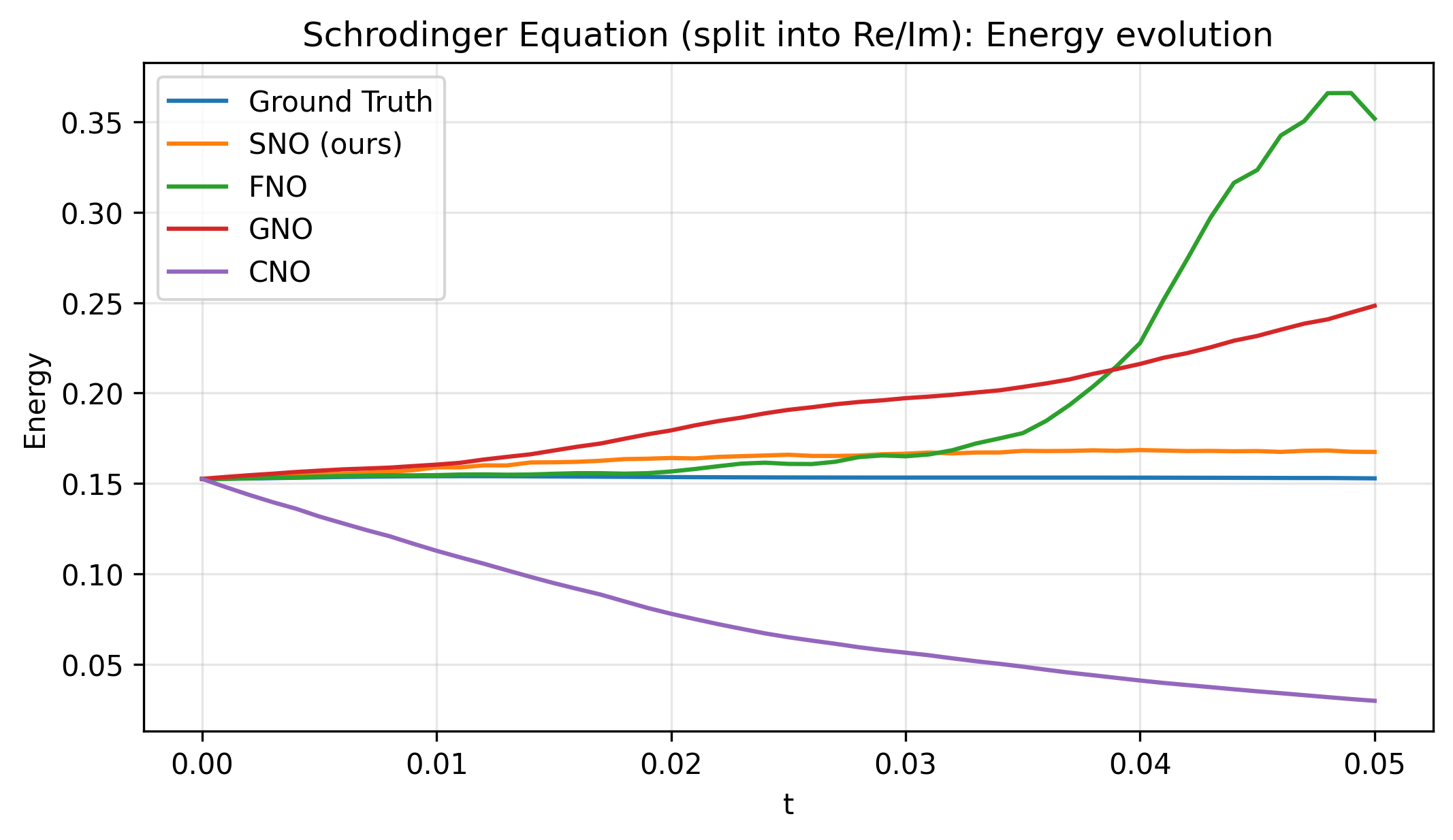}
  \caption{Schr\"odinger equation: Energy evolution at first 50 time-steps.}
  \label{fig:sch-energy}
\end{figure}

\begin{figure}[htbp]
  \centering
  \includegraphics[width=\linewidth]{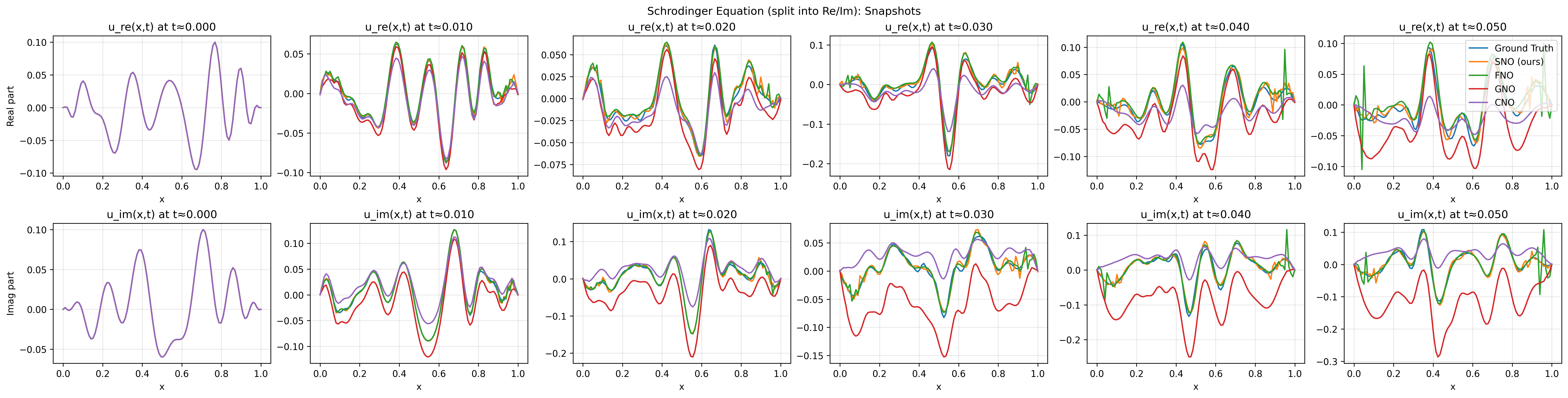}
  \caption{Schr\"odinger equation: snapshots of real part and imaginary aprt  at first 50 time-steps}
  \label{fig:sch-snapshots}
\end{figure}

\begin{figure}[htbp]
  \centering
  \includegraphics[width=\linewidth]{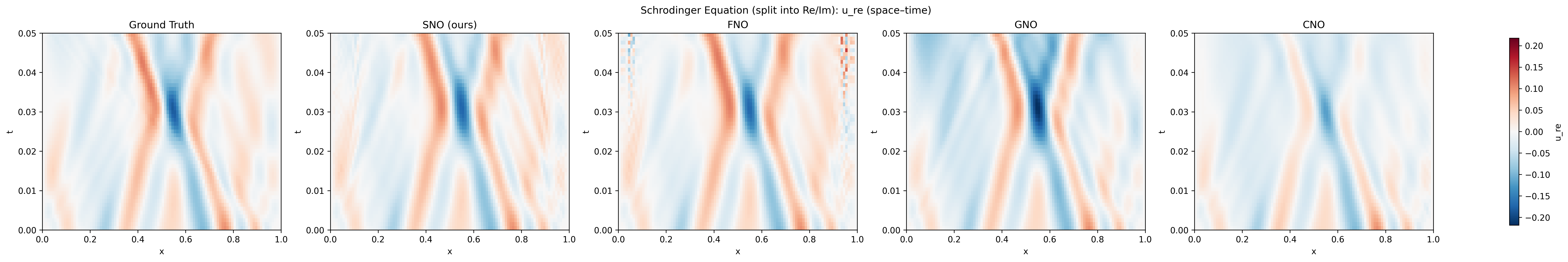}
  \caption{Schr\"odinger equation: space--time real part}
  \label{fig:sch-spacetime-re}
\end{figure}

\begin{figure}[htbp]
  \centering
  \includegraphics[width=\linewidth]{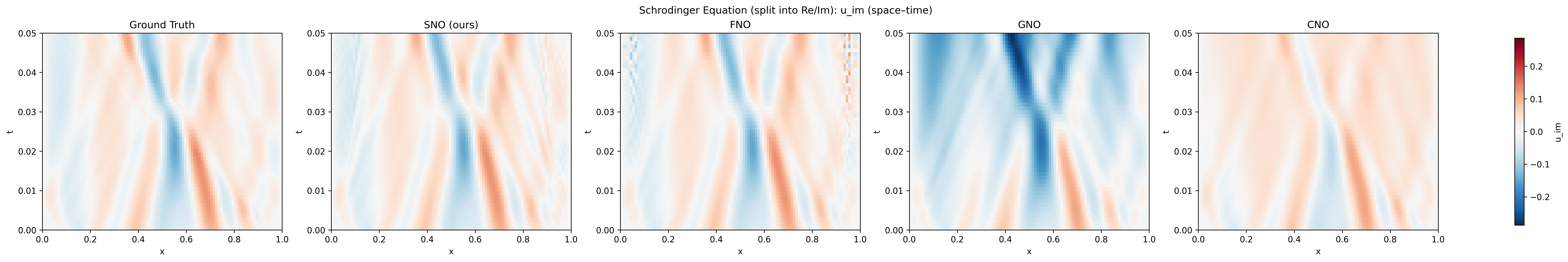}
  \caption{Schr\"odinger equation: space--time imaginary part}
  \label{fig:sch-spacetime-im}
\end{figure}

\begin{figure}[htbp]
  \centering
  \includegraphics[width=\linewidth]{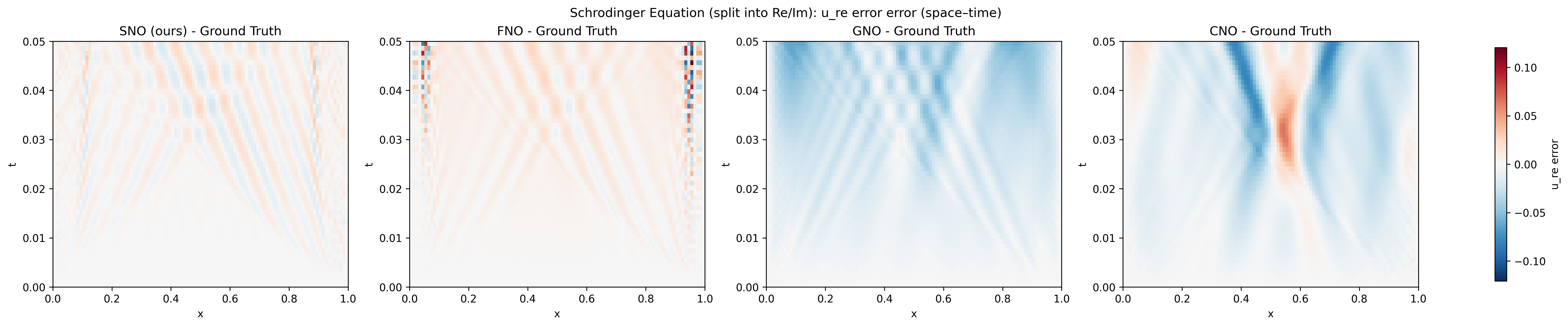}
  \caption{Schr\"odinger equation: space--time real part error}
  \label{fig:sch-spacetime-error-re}
\end{figure}
\begin{figure}[htbp]
  \centering
  \includegraphics[width=\linewidth]{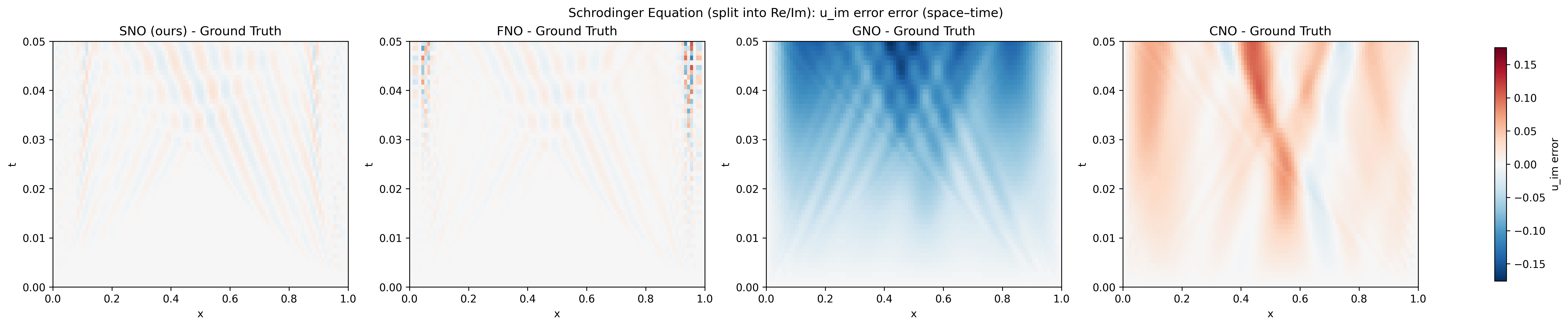}
  \caption{Schr\"odinger equation: space--time imaginary error
  }
  \label{fig:sch-spacetime-error-im}
\end{figure}

\begin{figure}[htbp]
  \centering
  \includegraphics[width=0.6\linewidth]{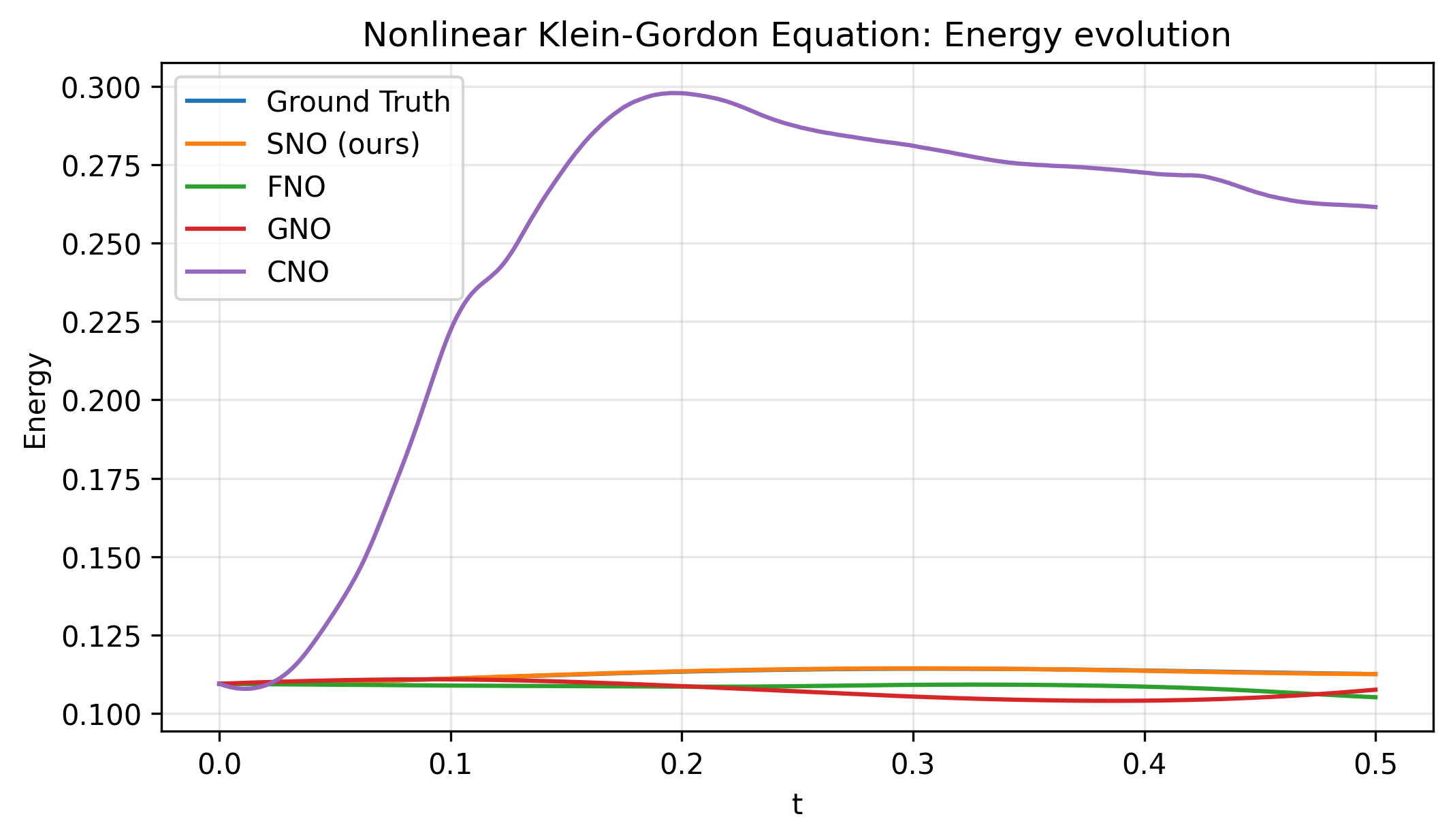}
  \caption{Nonlinear Klein--Gordon equation: Hamiltonian energy evolution.
  The baseline FNO exhibits rapid energy growth and eventual blow-up, whereas
  the Symplectic Neural Operator (SNO) maintains bounded energy evolution close
  to the finite-difference (FD) reference.}
  \label{fig:nkg-energy}
\end{figure}

\begin{figure}[htbp]
  \centering
  \includegraphics[width=\linewidth]{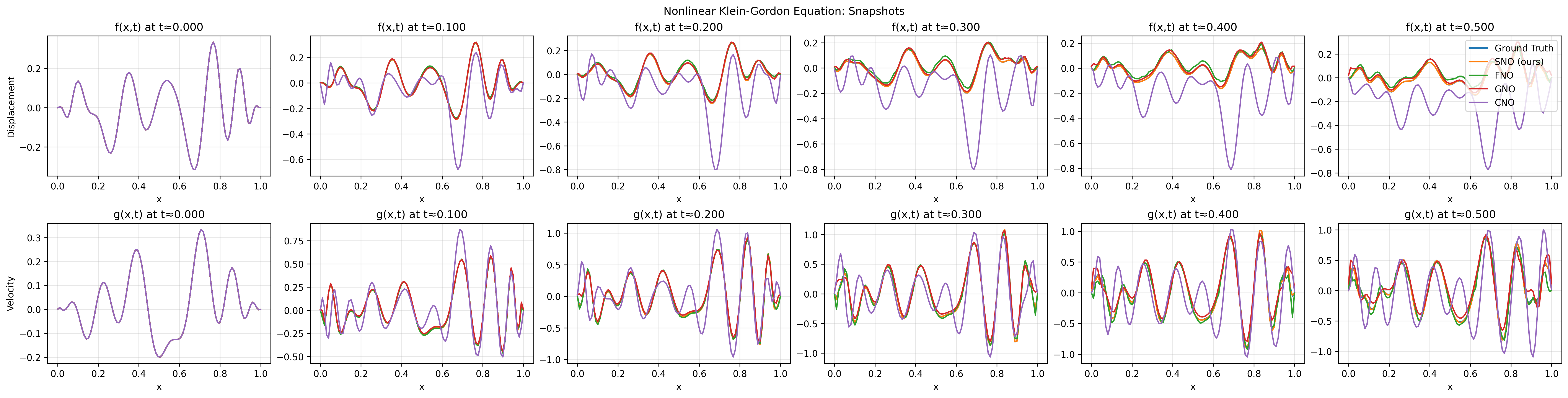}
  \caption{Nonlinear Klein--Gordon equation: snapshots of displacement
  $f(x,t)$ (top row) and velocity $g(x,t)$ (bottom row) at representative
  times.}
  \label{fig:nkg-snapshots}
\end{figure}

\begin{figure}[htbp]
  \centering
  \includegraphics[width=\linewidth]{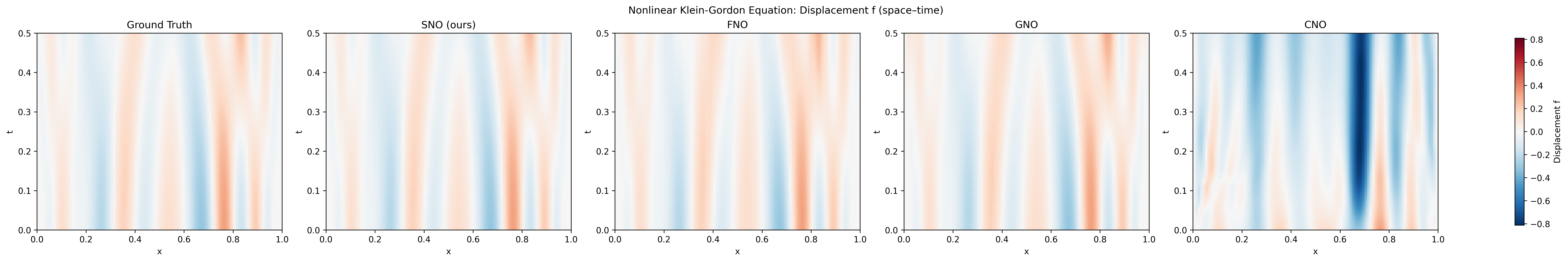}
  \caption{Nonlinear Klein--Gordon equation: space--time displacement field.}
  \label{fig:nkg-spacetime-f}
\end{figure}

\begin{figure}[htbp]
  \centering
  \includegraphics[width=\linewidth]{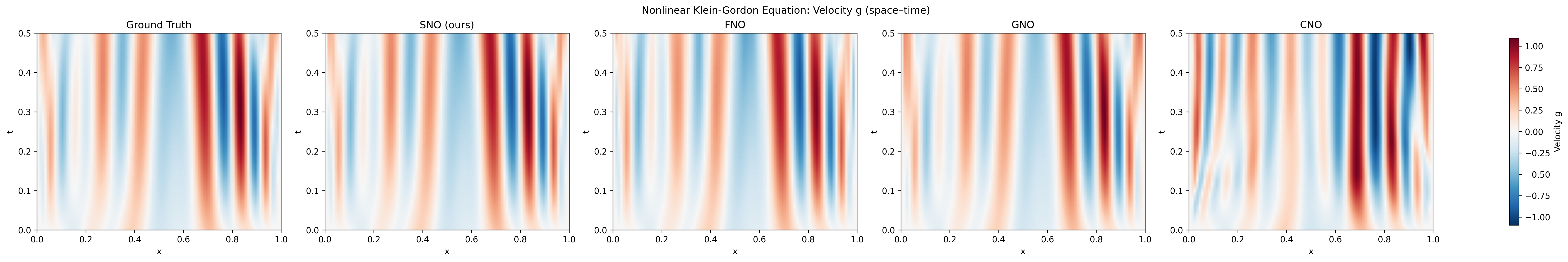}
  \caption{Nonlinear Klein--Gordon equation: space--time velocity field.}
  \label{fig:nkg-spacetime-g}
\end{figure}

\begin{figure}[htbp]
  \centering
  \includegraphics[width=0.95\linewidth]{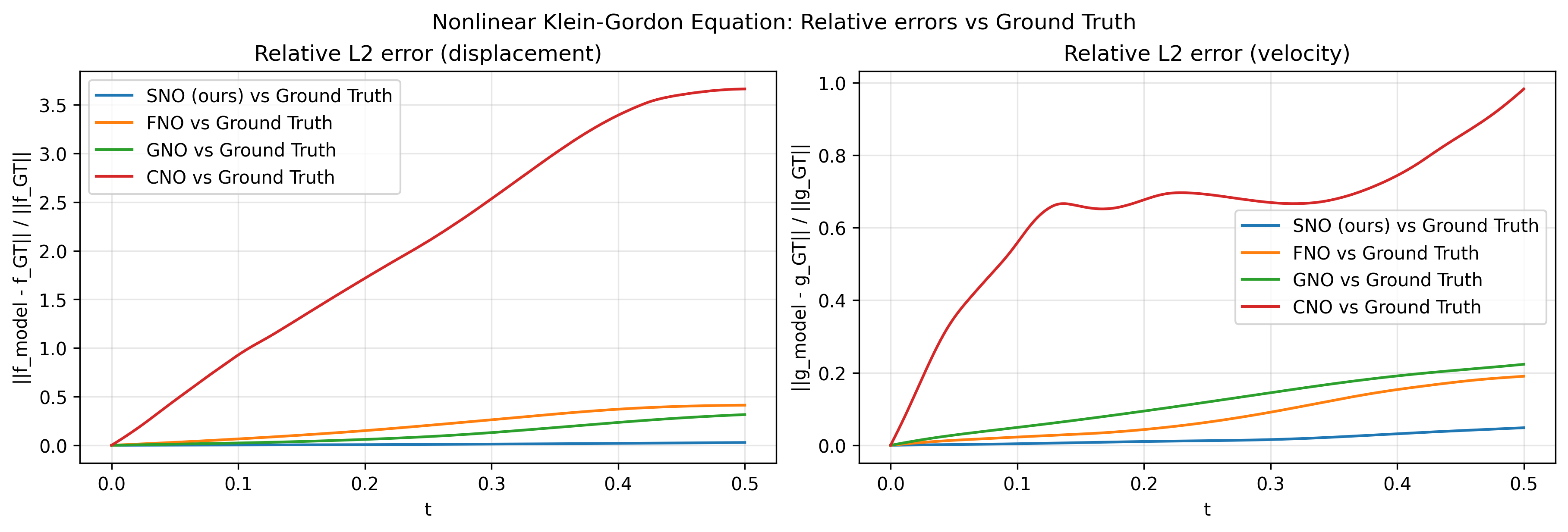}
  \caption{Klein-Gordon equation: relative $L^2$ rollout errors against the ground truth for displacement (left) and velocity (right).
  The self-adjoint SNO significantly reduces long-time error growth compared
  to the baseline models.}
  \label{fig:nkg-sno-relerr}
\end{figure}

\begin{figure}[htbp]
  \centering
  \includegraphics[width=\linewidth]{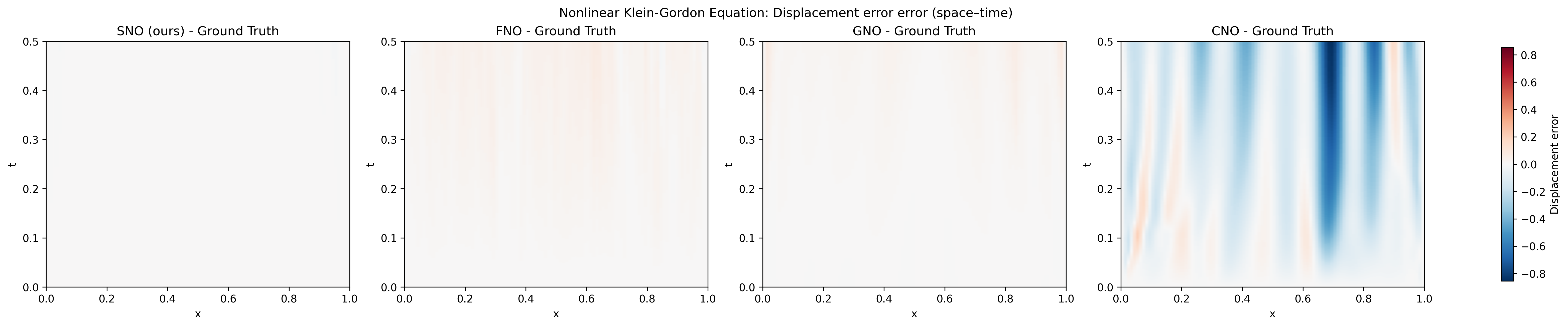}
  \caption{Nonlinear Klein--Gordon equation: space--time displacement error
  relative to ground truth.}
  \label{fig:nkg-spacetime-error-f}
\end{figure}

\begin{figure}[htbp]
  \centering
  \includegraphics[width=\linewidth]{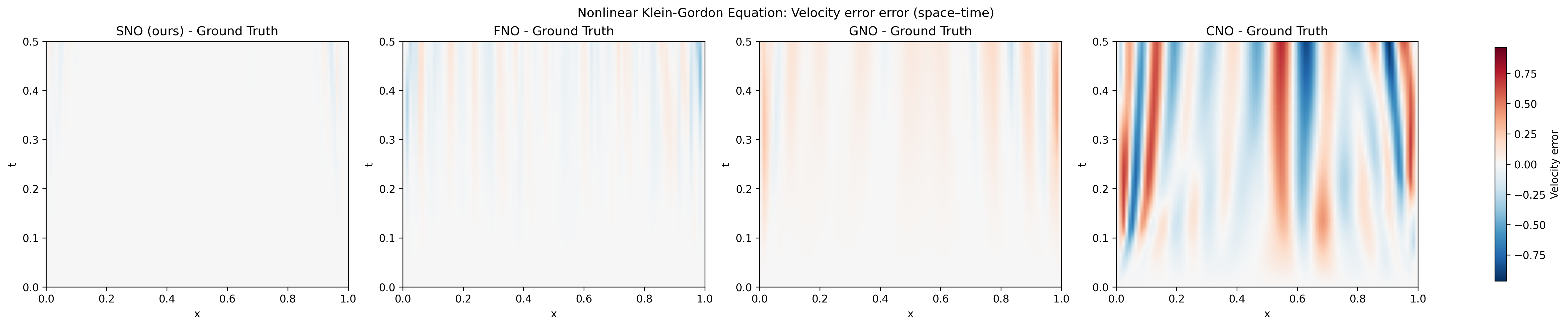}
  \caption{Nonlinear Klein--Gordon equation: space--time velocity error
  relative to ground truth. SNO exhibits significantly improved long-time stability
  compared to all baseline models.}
  \label{fig:nkg-spacetime-error-g}
\end{figure}



\end{document}